\documentclass{article}
\usepackage{amssymb}
\usepackage{amsmath}
\usepackage{amsfonts}
\usepackage{bbding}

\newtheorem{theorem}{Theorem}[section]

\newtheorem{corollary}[theorem]{Corollary}

\newtheorem{definition}[theorem]{Definition}

\newtheorem{lemma}[theorem]{Lemma}
\newtheorem{notation}[theorem]{Notation}

\newtheorem{remark}[theorem]{Remark}

\newenvironment{proof}[1][Proof]{\noindent\textbf{#1.} }{\ \rule{0.5em}{0.5em}}
\input{tcilatex}

\begin{document}

\begin{center}
\ \textbf{\large On solvability of a class of nonlinear elliptic type
equation with variable exponent}\ \ \ \ \ \ \ 

\bigskip

U\u{g}ur Sert and Kamal Soltanov \footnotetext{%
U. Sert(\Envelope): {}Faculty of Science, Department of Mathematics,
Hacettepe University, 06800, Beytepe, Ankara, Turkey. e-mail:
usert@hacettepe.edu.tr
\par
K. Soltanov: Faculty of Science, Department of Mathematics, Hacettepe
University, 06800, Beytepe, Ankara, Turkey. e-mail: soltanov@hacettepe.edu.tr%
}
\end{center}

\noindent \textbf{Abstract}. In this paper, {we study the Dirichlet problem
for the implicit degenerate nonlinear elliptic equation with variable
exponent in a bounded domain $\Omega \subset 
\mathbb{R}
^{n}$. We obtain sufficient conditions for the existence of a solution
without regularization and any restriction between the exponents.
Furthermore, we define the domain of the operator generated by posed problem
and investigate its some properties and also its relations with known spaces
that enable us to prove existence theorem.\ }

{\ \ \ \ \ \ \ \ \ \ \ \ \ \ \ \ \ \ \ \ \ \ \ \ \ \ \ \ \ \ \ \ \ \ \ \ \ \
\ \ \ \ \ \ \ \ \ \ \ \ \ \ \ \ \ \ \ \ \ \ \ \ \ \ \ \ \ \ \ \ \ \ \ \ \ \
\ \ \ \ \ \ \ \ \ \ \ \ \ \ \ \ \ \ \ \ \ \ \ \ \ \ \ \ \ \ \ \ \ }


\ \ \ \ \ \ \ \ \ \ \ \ \ \ \ \ \ \ \ \ \ \ \ \ \ \ \ \ \ \ \ \ \ \ \ \ \ \
\ \ \ \ \ \ \ \ \ \ \ \ \ \ \ \ \ \ \ \ \ \ \ \ \ \ \ \ \ \ \ \ \ \ \ \ \ \
\ \ \ \ \ \ \ \ \ \ \ \ \ \ \ \ \ \ \ \ \ \ \ 


\section{\protect\large Introduction}

In this paper, we study the Dirichlet problem for the nonlinear elliptic
equation with variable nonlinearity in a bounded domain $\Omega \subset 
\mathbb{R}
^{n}\left( n\geq 3\right) $ which has sufficently smooth boundary. 
\begin{equation}
\left\{ 
\begin{array}{l}
-\ \Delta \left( \left\vert u\right\vert ^{p(x)-2}u\right) +a\left(
x,u\right) =h\left( x\right) \\ 
\text{ \ \ }u\mid _{\partial \Omega }=0%
\end{array}%
\right.  \tag{1.1}
\end{equation}%
Here $p:$ $\Omega \longrightarrow 
\mathbb{R}
$, $2\leq p^{-}\leq p\left( x\right) \leq p^{+}<\infty $ and $p\in $ $%
C^{1}\left( \bar{\Omega}\right) $. Also the function $a:\Omega \times 
\mathbb{R}
\rightarrow 
\mathbb{R}
$, $a\left( x,\tau \right) $ has a variable nonlinearity up to $\tau $ (for
example $a\left( x,u\right) $ shall be in the form such as $a\left(
x,u\right) =a_{0}\left( x,u\right) \left\vert u\right\vert ^{\xi \left(
x\right) -1}u+a_{1}\left( x,u\right) $, see Section 4).

In recent years, there has been an increasing interest in the study of
equations with variable exponents of nonlinearities. The interest in the
study of differential equations that involves variable exponents is
motivated by their applications to the theory of elasticity and
hydrodynamics, in particular the models of electrorheological fluids [8, 25]
in which a substantial part of viscous energy, the thermistor problem [34],
image processing [9] and modelling of non-Newtonian fluids with
thermo-convective effects [5] etc.

The main feature in the equation%
\begin{equation}
-\ \Delta \left( \left\vert u\right\vert ^{\alpha (x)-2}u\right) +a\left(
x,u\right) =h\left( x\right)  \tag{1.0}
\end{equation}%
is clearly the exponential nonlinearity with respect to the solution that
makes it implicit degenerate. Such equations may appear, for instance, in
the mathematical description of the process of nonstable filtration of an
ideal barotropic gas in a nonhomogeneous porous medium. The equation of
state of the gas has the form $p=\rho ^{\alpha \left( x\right) }$ where $p$
is the pressure, $\rho $ is the density, and the exponent\ $\alpha \left(
x\right) $ is a given function then by using the known physical laws in that
case, we obtain an equation in the form of (1.0) (for sample see [4]). For
the several of the most important applications of nonlinear partial
differential equations with variable exponent arise from mathematical
modelling of suitable processes in mechanics, mathematical physics, image
processings etc., we refer to [22].(see also [20, 21, 25])

For some cases in gas dynamics as mentioned above, Lagrangian function $f$
in the definition of integral functional 
\begin{equation*}
F_{\Omega }\left( u\right) =\int\limits_{\Omega }f\left( x,u,Du\right) dx
\end{equation*}%
may satisfy the general nonstandard growth condition of the type 
\begin{equation*}
c_{0}\left\vert \tau \right\vert ^{m\left( x\right) }-a\left\vert
y\right\vert ^{\xi \left( x\right) }-g\left( x\right) \leq f\left( x,y,\tau
\right) \leq c_{1}\left\vert \tau \right\vert ^{m_{0}\left( x\right)
}+a\left\vert y\right\vert ^{\xi \left( x\right) }+g\left( x\right)
\end{equation*}%
with $1<$ $m\left( x\right) \leq m_{0}\left( x\right) $ and $m\left(
x\right) \leq \xi \left( x\right) $ where all exponents are continous
functions over $\bar{\Omega}$. In [36] Zhikov gives an example which shows
that if $f$ satisfies such type of inequality then appropriate functional
defined by $f$ may have the Lavrentiev phenomenon, the minimizer of the
functional is irregular. As known, it has important applications in
mechanics, there are too many papers in variational problems which has been
devoted to the case that $f$ holds this type condition [1, 15, 19, 36].

Also we note that the relation between the weak solutions of the class of
elliptic equations $-$div$A\left( x,u,Du\right) =B\left( x,u,Du\right) $ and
minimizer of the functional $F_{\Omega }$ under the nonstandard growth
condition given above was studied in [2, 13].

Recently, problems similar type to (1.1) have been studied in a lot of
papers [2, 5-7, 22, 25, 34-36]. In [35] Zhikov investigated elliptic problem
such as 
\begin{equation*}
\left\{ 
\begin{array}{l}
\Delta _{p\left( x\right) }u=\text{div}g \\ 
u\mid _{\partial \Omega }=0%
\end{array}%
\right. 
\end{equation*}%
here $g\in \left( L^{\infty }\left( \Omega \right) \right) ^{d}$, $\Delta
_{p\left( x\right) }$ is $p\left( x\right) $-Laplacian and $\Omega \subset 
\mathbb{R}
^{d}$ is bounded Lipschitz domain. He established the weak solution of the
considered problem with using the sequence of solutions of the problems
which converges to considered problem. Also some applications on solvability
analysis of a well-known coupled system in non-Newtonian hydrodynamics
without resorting to any smallness conditions was given in that paper.

In [9] authors have studied the problem related to image recovery. To
investigate that problem they first considered the elliptic part of the
problem namely they investigated the minimization problem 
\begin{equation*}
\underset{u\in BV\cap L^{2}\left( \Omega \right) }{\text{min}}%
\int\limits_{\Omega }\phi \left( x,Du\right) +\frac{\lambda }{2}\left(
u-I\right) ^{2},
\end{equation*}%
and proved the existence of solution in more general class by using
variational method. Here $BV$ is the space of functions of bounded variation
defined on $\Omega $.

In [7] authors have considered the Dirichlet boundary value problem for the
elliptic equation 
\begin{equation*}
-\sum\limits_{i}D_{i}\left( a_{i}\left( x,u\right) \left\vert
D_{i}u\right\vert ^{p_{i}\left( x\right) -2}D_{i}u\right) +c\left(
x,u\right) \left\vert u\right\vert ^{\sigma (x)-2}u=f\left( x\right)
\end{equation*}%
Under sufficent conditions they showed the existence of weak solution by
using Browder-Minty theorem for the special case $a_{i}\left( x,u\right)
\equiv A_{i}\left( x\right) $, $c\left( x,u\right) \equiv C\left( x\right) $%
. In the general case, the solution was constructed via Galerkin's method
under additional conditions for $a_{i}\left( x,u\right) $ and $c\left(
x,u\right) $.

In appearance in most of these papers we mentioned above, authors have
studied the problems which involves $p\left( .\right) $-Laplacian type
equation and used monotonicity methods. To the best of our knowledge, by now
there are no results on the existence of solutions to the elliptic equations
of the type (1.1) with nonconstant exponents of nonlinearity. However
similar type problem to (1.1) was studied in [6] and authors investigated
the regularized problem to show the existence of weak solution. In the
present paper, we investigate the problem (1.1) without regularization. We
also note that earlier Dubinskii [10] investigated problems which are
similar to (1.1) for constant exponents and obtained existence results.
Afterward, Raviart [24] obtained some results on uniqueness of solution for
this type problems.

Here we prove the existence of sufficently smooth, in some sense, solution
of the problem (1.1). Unlike the above papers, we investigate (1.1) without
monotonicity type conditions. Since we consider the posed problem under more
general (weak) conditions, in that case any method which is related to
monotonicity can not be used. Therefore we use a different method to
investigate the problem (1.1). We show that considered problem is
homeomorphic to the following problem:%
\begin{equation}
\left\{ 
\begin{array}{l}
\sum\limits_{i=1}^{n}-D_{i}\left( \left\vert u\right\vert
^{p_{0}-2}D_{i}u\right) +c\left( x,u\right) =h\left( x\right) \\ 
u\mid _{\partial \Omega }=0%
\end{array}%
\right.  \tag{1.2}
\end{equation}%
(see Section 3) and using this fact, we obtain existence of solution of
problem (1.1) (Section 4).

Moreover we study the posed problem in the space, that generated by this
problem. Investigating most of boundary value problem on its own space leads
to obtain better results. Henceforth here considered problem is investigated
on its own space. Unlike linear boundary value problems, the sets generated
by nonlinear problems are subsets of linear spaces, but not possessing the
linear structure [26-33].

This paper is organized as follows: In the next section, we recall some
useful results on the generalized Orlicz-Lebesgue spaces (Subsec. 2.1) and
results on nonlinear spaces (pn-spaces) (Subsec. 2.2). In Section 3, under
the sufficent conditions we show the existence of weak solution for the
problem (1.2). In Section 4, we give some additional results which are
required for existence theorem (Subsec. 4.1) and prove existence of a
generalized solution for the main problem (1.1) (Subsec. 4.2).

\section{\protect\large Preliminaries}

\subsection{\protect\normalsize Generalized Lebesgue spaces}

\bigskip In this subsection, some available facts from the theory of the
generalized Lebesgue spaces also called Orlicz-Lebesgue spaces will be
introduced. We present these facts without proofs which can be found in [
11, 12, 16, 17, 21].

Let $\Omega $ be a Lebesgue measurable subset of $%
\mathbb{R}
^{n}$ such that $mes\left( \Omega \right) >0$. (Throughout this paper, we
denote by $mes\left( \Omega \right) $ the Lebesgue measure of $\Omega $). By 
$P\left( \Omega \right) $ we denote the family of all measurable functions $%
p:\Omega \longrightarrow \left[ 1,\infty \right] .$

For $p\in $ $P\left( \Omega \right) ,$ $\Omega _{\infty }^{p}\equiv \Omega
_{\infty }\equiv \left\{ x\in \Omega |\text{ }p\left( x\right) =\infty
\right\} $ then on the set of all functions on $\Omega $ define the
functional $\sigma _{p}$ and $\left\Vert .\right\Vert _{p}$ by%
\begin{equation*}
\sigma _{p}\left( u\right) \equiv \int\limits_{\Omega \backslash \Omega
_{\infty }}\left\vert u\right\vert ^{p\left( x\right) }dx+\underset{\Omega
_{\infty }}{ess}\sup \left\vert u\left( x\right) \right\vert
\end{equation*}%
and%
\begin{equation*}
\left\Vert u\right\Vert _{L^{p\left( x\right) }\left( \Omega \right) }\equiv
\inf \left\{ \lambda >0|\text{ }\sigma _{p}\left( \frac{u}{\lambda }\right)
\leq 1\right\} .
\end{equation*}%
Clearly if $p\in L^{\infty }\left( \Omega \right) $ then%
\begin{equation*}
1\leq p^{-}\equiv \underset{\Omega }{ess}\inf \left\vert p\left( x\right)
\right\vert \leq \underset{\Omega }{ess}\sup \left\vert p\left( x\right)
\right\vert \equiv p^{+}<\infty
\end{equation*}%
in that case we have 
\begin{equation*}
\sigma _{p}\left( u\right) \equiv \int\limits_{\Omega }\left\vert
u\right\vert ^{p\left( x\right) }dx\text{.}
\end{equation*}%
The Generalized Lebesgue space is defined as follows:%
\begin{equation*}
L^{p\left( x\right) }\left( \Omega \right) \equiv \left\{ u|\text{ }\exists
\lambda >0,\text{ }\sigma _{p}\left( \lambda u\right) <\infty \right\}
\end{equation*}%
The space $L^{p\left( x\right) }\left( \Omega \right) $ becomes a Banach
space under the norm $\left\Vert .\right\Vert _{L^{p\left( x\right) }\left(
\Omega \right) }$which is so-called Luxemburg norm.

Let $\Omega \subset 
\mathbb{R}
^{n}$ be a bounded domain and $p\in L^{\infty }\left( \Omega \right) $ then
Generalized Sobolev space is defined as follows:%
\begin{equation*}
W^{m,p\left( x\right) }\left( \Omega \right) \equiv \left\{ u\in L^{p\left(
x\right) }\left( \Omega \right) |\text{ }D^{\alpha }u\in L^{p\left( x\right)
}\left( \Omega \right) ,\text{ }\left\vert \alpha \right\vert \leq m\right\}
\end{equation*}%
and this space is separable Banach space under the norm:%
\begin{equation*}
\left\Vert u\right\Vert _{W^{m,p\left( x\right) }\left( \Omega \right)
}\equiv \sum\limits_{\left\vert \alpha \right\vert \leq m}\left\Vert
D^{\alpha }u\right\Vert _{L^{p\left( x\right) }\left( \Omega \right) }
\end{equation*}%
The following results are known for these spaces: [3, 11, 17, 23].

\begin{lemma}
Let $0<mes\left( \Omega \right) <\infty ,$ and $p_{1},$ $p_{2}\in P\left(
\Omega \right) $ then%
\begin{equation*}
L^{p_{1}\left( x\right) }\left( \Omega \right) \subset L^{p_{2}\left(
x\right) }\left( \Omega \right) \iff \text{ }p_{2}\left( x\right) \leq
p_{1}\left( x\right) \text{ for a.e }x\in \Omega
\end{equation*}
\end{lemma}

\begin{lemma}
The dual space to $L^{p\left( x\right) }\left( \Omega \right) $ is $%
L^{p^{\ast }\left( x\right) }\left( \Omega \right) $ if and only if $p\in
L^{\infty }\left( \Omega \right) $. The space $L^{p\left( x\right) }\left(
\Omega \right) $ is reflexive if and only if%
\begin{equation*}
1<p^{-}\leq p^{+}<\infty
\end{equation*}%
here 
\begin{equation*}
p^{\ast }\left( x\right) \equiv \left\{ 
\begin{array}{l}
\infty \text{ \ \ \ \ for }x\in \Omega _{1}^{p} \\ 
1\text{ \ \ \ \ \ for }x\in \Omega _{\infty }^{p} \\ 
\frac{p\left( x\right) }{p\left( x\right) -1}\text{ for other }x\in \Omega%
\end{array}%
\right.
\end{equation*}
\end{lemma}

\begin{lemma}
Let $p,q\in C\left( \bar{\Omega}\right) $ and $p,q\in L^{\infty }\left(
\Omega \right) .$ Assume that%
\begin{equation*}
mp\left( x\right) <n,\text{ }q\left( x\right) <\frac{np\left( x\right) }{%
n-mp\left( x\right) }\text{, }\forall x\in \bar{\Omega}
\end{equation*}%
Then there is a continous and compact embedding $W^{m,p\left( x\right)
}\left( \Omega \right) \hookrightarrow L^{q\left( x\right) }\left( \Omega
\right) .$
\end{lemma}

\bigskip

\subsection{\protect\bigskip {\protect\normalsize On pn-spaces}}

In this subsection, we introduce some function classes which are complete
metric spaces and directly connected to the considered problem. Also we give
some embedding results for these spaces [28, 32] (see also [31, 26, 27, 29,
33]).

\begin{definition}
Let $\alpha \geq 0,$ $\beta \geq 1$, $\varrho =\left( \varrho
_{1,..,}\varrho _{n}\right) $ is multi-index, $m\in 
\mathbb{Z}
^{+},$ $\Omega \subset 
\mathbb{R}
^{n}\left( n\geq 1\right) $ is bounded domain with sufficently smooth
boundary.%
\begin{equation*}
S_{m,\alpha ,\beta }\left( \Omega \right) \equiv \left\{ u\in L_{1}\left(
\Omega \right) \mid \left[ u\right] _{S_{m,\alpha ,\beta }\left( \Omega
\right) }^{\alpha +\beta }\equiv \sum_{0\leq \left\vert \varrho \right\vert
\leq m}\left( \int\limits_{\Omega }\left\vert u\right\vert ^{\alpha
}\left\vert D^{\varrho }u\right\vert ^{\beta }dx\right) <\infty \right\}
\end{equation*}%
in particularly, 
\begin{equation*}
\mathring{S}_{1,\alpha ,\beta }\left( \Omega \right) \equiv \left\{ u\in
L_{1}\left( \Omega \right) \mid \left[ u\right] _{S_{1,\alpha ,\beta }\left(
\Omega \right) }^{\alpha +\beta }\equiv \sum_{i=1}^{n}\left(
\int\limits_{\Omega }\left\vert u\right\vert ^{\alpha }\left\vert
D_{i}u\right\vert ^{\beta }dx\right) <\infty \right\} \cap \left\{ u\mid
_{\partial \Omega }\equiv 0\right\}
\end{equation*}%
These spaces are called pn-spaces.\footnote{$S_{1,\alpha ,\beta }\left(
\Omega \right) $ is metric space with the following metric: $\forall u,v\in
S_{1,\alpha ,\beta }\left( \Omega \right) $%
\par
\begin{equation*}
d_{S_{1,\alpha ,\beta }}\left( u,v\right) =\left\Vert \left\vert
u\right\vert ^{\frac{\alpha }{\beta }}u-\left\vert v\right\vert ^{\frac{%
\alpha }{\beta }}v\right\Vert _{W^{1,\beta }\left( \Omega \right) }
\end{equation*}%
}
\end{definition}

\begin{theorem}
Let $\alpha \geq 0,$ $\beta \geq 1$ then $\varphi :%
\mathbb{R}
\longrightarrow 
\mathbb{R}
$, $\varphi \left( t\right) \equiv \left\vert t\right\vert ^{\frac{\alpha }{%
\beta }}t$ is a homeomorphism between $S_{1,\alpha ,\beta }\left( \Omega
\right) $ and $W^{1,\beta }\left( \Omega \right) $.
\end{theorem}

\begin{theorem}
(i) Let $\alpha ,$ $\alpha _{1}\geq 0$ and $\beta _{1}\geq 1$, $\beta \geq
\beta _{1},$ $\frac{\alpha _{1}}{\beta _{1}}\geq \frac{\alpha }{\beta },$ $%
\alpha _{1}+\beta _{1}\leq \alpha +\beta $ then we have%
\begin{equation*}
\mathring{S}_{1,\alpha ,\beta }\left( \Omega \right) \subseteq \mathring{S}%
_{1,\alpha _{1},\beta _{1}}\left( \Omega \right)
\end{equation*}%
\ \ \ (ii) Let $\alpha \geq 0,$ $\beta \geq 1,$ $n>\beta $ and $\frac{%
n\left( \alpha +\beta \right) }{n-\beta }\geq r$ then there is a continous
embedding%
\begin{equation*}
\mathring{S}_{1,\alpha ,\beta }\left( \Omega \right) \subset L^{r}\left(
\Omega \right) \text{ }
\end{equation*}%
Furthermore for $\frac{n\left( \alpha +\beta \right) }{n-\beta }>r$ the
embedding is compact.

(iii) If $\alpha \geq 0,$ $\beta \geq 1$ and $p\geq \alpha +\beta \ $then%
\begin{equation*}
W_{0}^{1,p}\left( \Omega \right) \subset \mathring{S}_{1,\alpha ,\beta
}(\Omega )
\end{equation*}%
is hold.
\end{theorem}

\bigskip

Now we give a general result [30] (see also for similar theorems [26, 27,
31, 33]) that will be used to show existence of weak solution of the posed
problem (1.1) .

\begin{definition}
Let $X,$ $Y$ be Banach spaces, $Y^{\ast }$ is the dual space of $Y$ and $%
S_{0}$ is a weakly complete pn-space. $f:S_{0}\subset X\longrightarrow Y$ a
nonlinear mapping. $f$ \ is a \textquotedblleft coercive\textquotedblright\
operator in a generalized sense if there exists a bounded operator $%
g:X_{0}\subseteq S_{0}\longrightarrow Y^{\ast }$ that satisfies the
conditions, $\overline{X_{0}}=S_{0},$ $\overline{{Im}g}=Y^{\ast }$ and a
continuous function $\mu :%
\mathbb{R}
^{+}\longrightarrow 
\mathbb{R}
$ non-decrasing such that the following relation is valid for a dual form $%
\langle .,.\rangle $ with respect to the pair of spaces ($Y$, $Y^{\ast }$): 
\begin{equation*}
\langle f\left( x\right) ,g\left( x\right) \rangle \geq \mu \left( \left[ x%
\right] _{S_{0}}\right) \text{ for }x\in X_{0}\text{ and }\exists r>0\ni \mu
\left( r\right) \geq 0
\end{equation*}%
In this case it is said that the mappings $f$ and $g$ generate a "coercive
pair " on $X_{0}.$
\end{definition}

\begin{definition}
Let $X_{0}$ be a topological space such that $X_{0}\subset $ $S_{0}\subset $ 
$X$, and let $f$ \ be a nonlinear mapping acting from $X$ \ to $Y$ where $Y$
is a reflexive space such that both $Y$ and $Y^{\ast }$ are strictly convex.
An element $x\in S_{0}$ satisfying 
\begin{equation}
\langle f\left( x\right) ,y^{\ast }\rangle =\langle y,y^{\ast }\rangle \text{%
, }\forall y^{\ast }\in M^{\ast }\subseteq Y^{\ast }\text{, }y\in Y 
\tag{2.1}
\end{equation}
the equation (2.1) is called a $M^{\ast }$-solution of the equation $f\left(
x\right) =y$.
\end{definition}

We will consider the following conditions:\ \ \ \ \ \ \ \ \ \ \ \ \ \ \ \ \
\ \ \ \ \ \ \ \ \ \ \ \ \ \ \ \ \ \ \ \ \ \ \ \ \ \ \ \ \ \ \ \ \ \ \ 

(a) $f:S_{0}\longrightarrow Y$ $\ $is a weakly compact (weakly continous)
mapping and there exists a closed linear subspace $Y_{0}$ of $Y$ such that $%
f:S_{0}\longrightarrow Y_{0}\subseteq Y.$

(b) There exist a mapping $g:X_{0}\subset S_{0}\longrightarrow Y^{\ast }$
such that $g\left( X_{0}\right) $ contains a linear manifold from $Y^{\ast }$
which is dense in a closed linear subspace $Y_{0}^{\ast }$ of $Y^{\ast }$
and generates a "coercive pair" with $f$ on $X_{0}$ in a generalized sense.

Moreover one of the following conditions (1) or (2) hold:

(1) If $\ g$ is a linear continous operator then $S_{0}$ is a "reflexive"
space [27]. $X_{0}$ is a separable vector topological space which is dense
in $S_{0}$ and $\ker g^{\ast }=\left\{ 0\right\} $(where $g^{\ast }$ denotes
the adjoint of the linear continous operator $g$).

(2) If $\ g$ is a nonlinear operator then $Y_{0}^{\ast }$ is a separable
subspace of $Y^{\ast }$ and $g^{-1}$ is weakly continous from $Y^{\ast }$ to 
$S_{0}.$

\begin{theorem}
\bigskip Let the conditions (a),(b) and either (1) or (2) hold. Furthermore,
assume that a set $Y_{0}\subseteq Y$ is given such that for each $y\in Y_{0}$
the following condition is satisfed: there exists $r=r\left( y\right) >0$
such that 
\begin{equation*}
\mu (\left[ x\right] _{S_{0}})\geq \langle y,g\left( x\right) \rangle \text{%
, }\forall x\in X_{0},\left[ x\right] _{S_{0}}\geq r\text{ }
\end{equation*}%
Then equation (2.1) is $Y_{0}^{\ast }$-solvable in $S_{0}$ for any $y$ from
the subset $Y_{0}$ of $Y.$
\end{theorem}

\section{\protect\large Existence Results for Problem (1.2)}

As mentioned in introduction, studying the existence of solution of the
problem (1.1) requires to investigate problem (1.2) therefore firstly we
give the existence results for problem

\begin{equation}
\left\{ 
\begin{array}{l}
\sum\limits_{i=1}^{n}-D_{i}\left( \left\vert u\right\vert
^{p_{0}-2}D_{i}u\right) +c\left( x,u\right) =h\left( x\right) \\ 
u\mid _{\partial \Omega }=0%
\end{array}%
\right.  \tag{1.2}
\end{equation}%
here $p_{0}\geq 2$ and $c:\Omega \times 
\mathbb{R}
\rightarrow 
\mathbb{R}
,$ $c\left( x,\tau \right) $ has a variable nonlinearity up to $\tau $ (see
inequality (3.1)). Let the function $c\left( x,\tau \right) $ in problem
(1.2) hold the following conditions:

\textbf{(i)} $c:\Omega \times 
\mathbb{R}
\longrightarrow 
\mathbb{R}
$\textit{\ is a Caratheodory function and for the measurable function }$%
\alpha :\Omega \longrightarrow 
\mathbb{R}
$\textit{\ which satisfy }$1<$\textit{\ }$\alpha ^{-}\leq \alpha \left(
x\right) \leq \alpha ^{+}<\infty $\textit{, }$c\left( x,\tau \right) $%
\textit{\ holds the inequality}%
\begin{equation}
\left\vert c\left( x,\tau \right) \right\vert \leq c_{0}\left( x\right)
\left\vert \tau \right\vert ^{\alpha \left( x\right) -1}+c_{1}\left(
x\right) ,\text{ }\left( x,\tau \right) \in \Omega \times 
\mathbb{R}
,  \tag{3.1}
\end{equation}%
\textit{here }$c_{0},$\textit{\ }$c_{1}$\textit{\ are nonnegative,
measurable functions defined on }$\Omega $\textit{.}

Since on different values of $\alpha ^{+}$ depending on $p_{0}$ and $n$ i.e. 
$\alpha ^{+}<p_{0}$, $p_{0}\leq \alpha ^{+}<\tilde{p}$ and $\tilde{p}\leq
\alpha ^{+}<\infty $ where $\tilde{p}$ is critical exponent in Theorem 2.6
(ii) different conditions is required because of the circumstances appearing
in the embedding theorems for these spaces hence we separate the domain $%
\Omega $ up to these cases to three disjoint sets say $\Omega _{1},$ $\Omega
_{2}$ and $\Omega _{3}.$ By doing this, we obtained more slightly sufficent
conditions to show the existence of weak solution.\footnote{%
Since $\Omega $ is separated to three disjoint subsets, in some sense, one
can consider that problem as unity of three different problems.}

\textit{Let }$\eta \in \left( 0,1\right) $\textit{\ is sufficently small and
we define the sets}%
\begin{eqnarray*}
\Omega _{1} &\equiv &\left\{ x\in \Omega |\text{ }\alpha \left( x\right) \in
\lbrack 1,p_{0}-\eta )\right\} \\
\Omega _{2} &\equiv &\left\{ x\in \Omega |\text{ }\alpha \left( x\right) \in
\lbrack p_{0}-\eta ,\tilde{p})\right\} \\
\Omega _{3} &\equiv &\left\{ x\in \Omega |\text{ }\alpha \left( x\right) \in
\lbrack \tilde{p},\alpha ^{+}]\right\}
\end{eqnarray*}%
\textit{here critical }$\tilde{p}>p_{0}$\textit{\ and will be defined later}.

\textbf{(ii)\ }\textit{There exist a measurable function }$\alpha
_{1}:\Omega _{2}\longrightarrow 
\mathbb{R}
$\textit{\ which satisfy }$1\leq $\textit{\ }$\alpha _{1}^{-}\leq \alpha
_{1}\left( x\right) \leq \alpha _{1}^{+}<p_{0}$\textit{, such that }$c\left(
x,\tau \right) $\textit{\ holds the inequality}%
\begin{equation}
c\left( x,\tau \right) \tau \geq -c_{2}\left( x\right) \left\vert \tau
\right\vert ^{\alpha _{1}\left( x\right) }-c_{3}\left( x\right) ,\text{ }%
\left( x,\tau \right) \in \Omega _{2}\times 
\mathbb{R}
,  \tag{3.2}
\end{equation}%
\textit{here }$c_{2},$\textit{\ }$c_{3}$\textit{\ are nonnegative,
measurable functions defined on }$\Omega _{2}$\textit{.}

\textbf{(iii)} \textit{On }$\Omega _{3}\times 
\mathbb{R}
$\textit{, }$c\left( x,\tau \right) $\textit{\ satisfies the inequality}%
\begin{equation}
c\left( x,\tau \right) \tau \geq c_{4}\left( x\right) \left\vert \tau
\right\vert ^{\alpha \left( x\right) }-c_{5}\left( x\right) ,\text{ }\left(
x,\tau \right) \in \Omega _{3}\times 
\mathbb{R}
,  \tag{3.3}
\end{equation}%
\textit{here }$\alpha $\textit{\ is the same function as in (3.1)\ and\ }$%
c_{4}\left( x\right) \geq \bar{C}_{0}>0$\textit{\ a.e. }$x\in \Omega _{3}$%
\textit{\ and }$c_{4},$\textit{\ }$c_{5}$\textit{\ are nonnegative,
measurable functions defined on }$\Omega _{3}$\textit{.}

\bigskip

We consider the problem (1.2) for the functions $h\in W^{-1,q_{0}}\left(
\Omega \right) +L^{\alpha ^{\ast }\left( x\right) }\left( \Omega \right) $
where $\alpha ^{\ast }$ is conjugate of $\alpha $ i.e. $\alpha ^{\ast
}\left( x\right) \equiv \frac{\alpha \left( x\right) }{\alpha \left(
x\right) -1}$ and $q_{0}\equiv \frac{p_{0}}{p_{0}-1}$.

Let us define the following class of functions 
\begin{equation*}
Q_{0}\equiv \mathring{S}_{1,\left( p_{0}-2\right) q_{0},q_{0}}\left( \Omega
\right) \cap L^{\alpha \left( x\right) }\left( \Omega \right)
\end{equation*}

We understand the solution of the considered problem in the following sense.

\begin{definition}
A function $u\in Q_{0}$, is called the generalized solution (weak solution)
of problem (1.2) if it satisfies the equality%
\begin{equation*}
\sum_{i=1}^{n}\int\limits_{\Omega }\left( \left\vert u\right\vert
^{p_{0}-2}D_{i}u\right) D_{i}wdx+\int\limits_{\Omega }c\left( x,u\right)
wdx=\int\limits_{\Omega }hwdx
\end{equation*}%
for all $w\in W_{0}^{1,p_{0}}\left( \Omega \right) \cap L^{\alpha (x)}\left(
\Omega \right) $.\ \ \ \ \ \ \ \ \ \ \ \ \ \ \ \ \ \ \ \ \ \ \ \ \ 
\end{definition}

\begin{theorem}
Let \textbf{(i)}-\textbf{(iii)} hold. If $c_{2}\in L^{\frac{p_{0}}{%
p_{0}-\alpha _{1}\left( x\right) }}\left( \Omega _{2}\right) ,$ $c_{3}\in
L^{1}\left( \Omega _{2}\right) ,$ $c_{5}\in L^{1}\left( \Omega _{3}\right) ,$%
\ $c_{4}\in L^{\infty }\left( \Omega _{3}\right) ,$ $c_{1}\in L^{\beta
_{1}\left( x\right) }\left( \Omega \right) $, $c_{0}\in L^{\beta \left(
x\right) }\left( \Omega \right) $ where $\beta _{1}\left( x\right) \equiv
\left\{ 
\begin{array}{l}
\alpha ^{\ast }\left( x\right) \text{ if\ }x\in \Omega _{1} \\ 
q_{0}\text{ \ \ \ \ \ if\ }x\in \Omega _{2}\cup \Omega _{3}%
\end{array}%
\right. $ and $\beta \left( x\right) \equiv \left\{ 
\begin{array}{c}
\frac{p_{0}\alpha ^{\ast }\left( x\right) }{p_{0}-\alpha \left( x\right) }%
\text{ if }x\in \Omega _{1} \\ 
\frac{\tilde{p}\alpha ^{\ast }\left( x\right) }{\tilde{p}-\alpha \left(
x\right) }\text{ if }x\in \Omega _{2} \\ 
\text{ }\infty \text{ \ \ \ \ if }x\in \Omega _{3}%
\end{array}%
\right. $, $\tilde{p}\equiv \frac{np_{0}}{n-q_{0}}$, then $\forall h\in
W^{-1,q_{0}}\left( \Omega \right) +L^{\alpha ^{\ast }\left( x\right) }\left(
\Omega \right) $ problem (1.2) has a generalized solution in the space $%
Q_{0} $.\ \ \ \ \ \ \ \ \ \ \ \ \ \ \ \ \ \ \ \ \ \ \ \ \ \ \ \ \ \ \ \ \ \
\ \ \ \ \ \ \ \ \ \ \ \ \ \ \ \ \ \ \ \ \ \ \ \ \ \ \ \ \ \ \ \ \ \ \ \ \ \
\ \ \ \ \ \ \ \ \ \ \ \ \ \ \ \ \ \ \ \ \ \ \ \ \ \ \ \ \ \ \ \ \ \ \ \ \ \
\ \ \ \ \ \ \ \ \ \ \ \ \ 
\end{theorem}

\bigskip

The proof is based on Theorem 2.9. To use this, we introduce the following
spaces and mappings in order to apply Theorem 2.9 to prove Theorem 3.2. 
\begin{equation*}
S_{0}\equiv \mathring{S}_{1,\left( p_{0}-2\right) q_{0},q_{0}}\left( \Omega
\right) \cap L^{\alpha (x)}\left( \Omega \right) ,\text{ }Y\equiv
W^{-1,q_{0}}\left( \Omega \right) +L^{\alpha ^{\ast }\left( x\right) }\left(
\Omega \right) ,\text{ }X_{0}\equiv W_{0}^{1,p_{0}}\left( \Omega \right)
\cap L^{\alpha (x)}\left( \Omega \right)
\end{equation*}%
and%
\begin{equation*}
Y_{0}^{\ast }\equiv Y^{\ast }\equiv X_{0}
\end{equation*}

\begin{equation*}
f:S_{0}\longrightarrow Y
\end{equation*}%
\begin{equation}
f\left( u\right) \equiv \sum_{i=1}^{n}-D_{i}\left( \left\vert u\right\vert
^{p_{0}-2}D_{i}u\right) +c\left( x,u\right)  \tag{3.4}
\end{equation}%
\begin{equation*}
g:X_{0}\subset S_{0}\longrightarrow Y^{\ast }
\end{equation*}%
\begin{equation}
g\equiv Id  \tag{3.5}
\end{equation}

\bigskip

We prove some lemmas to show that all conditions of Theorem 2.9 are
fulfilled under the conditions of Theorem 3.2.

\begin{lemma}
Under the conditions of Theorem 3.2, the mappings $f$ and $g$ defined by
(3.4) and (3.5) respectively generate a "coercive pair" on $%
W_{0}^{1,p_{0}}\left( \Omega \right) \cap L^{\alpha (x)}\left( \Omega
\right) .$
\end{lemma}

\begin{proof}
Since $g\equiv Id,$ being "coercive pair" \ equals to order coercivity of $f$
on the space $W_{0}^{1,p_{0}}\left( \Omega \right) \cap L^{\alpha (x)}\left(
\Omega \right) $. For $u\in W_{0}^{1,p_{0}}\left( \Omega \right) \cap
L^{\alpha (x)}\left( \Omega \right) $%
\begin{equation*}
\langle f\left( u\right) ,u\rangle =\sum_{i=1}^{n}\left( \int\limits_{\Omega
}\left\vert u\right\vert ^{p_{0}-2}\left\vert D_{i}u\right\vert
^{2}dx\right) +\int\limits_{\Omega }c\left( x,u\right) udx
\end{equation*}%
\begin{equation*}
=\sum_{i=1}^{n}\left( \int\limits_{\Omega }\left\vert u\right\vert
^{p_{0}-2}\left\vert D_{i}u\right\vert ^{2}dx\right) +\int\limits_{\Omega
_{1}}c\left( x,u\right) udx+\int\limits_{\Omega _{2}}c\left( x,u\right)
udx+\int\limits_{\Omega _{3}}c\left( x,u\right) udx
\end{equation*}%
Using $(3.1)$, $(3.2)$, $(3.3)$, we obtain

\begin{equation*}
\geq \sum_{i=1}^{n}\left( \int\limits_{\Omega }\left\vert u\right\vert
^{p_{0}-2}\left\vert D_{i}u\right\vert ^{2}dx\right) -\int\limits_{\Omega
_{1}}\left\vert c_{0}\left( x\right) \right\vert \left\vert u\right\vert
^{\alpha \left( x\right) }dx-\int\limits_{\Omega _{1}}\left\vert c_{1}\left(
x\right) \right\vert \left\vert u\right\vert dx
\end{equation*}%
\begin{equation*}
-\int\limits_{\Omega _{2}}\left\vert c_{2}\left( x\right) \right\vert
\left\vert u\right\vert ^{\alpha _{1}\left( x\right) }dx-\int\limits_{\Omega
_{2}}\left\vert c_{3}\left( x\right) \right\vert dx+
\end{equation*}%
\begin{equation}
+\int\limits_{\Omega _{3}}\left\vert c_{4}\left( x\right) \right\vert
\left\vert u\right\vert ^{\alpha \left( x\right) }dx-\int\limits_{\Omega
_{3}}\left\vert c_{5}\left( x\right) \right\vert dx  \tag{3.6}
\end{equation}

Let estimate the second,third and fourth integrals in $(3.6)$ respectively.
For arbitrary $\epsilon _{i}>0$ ($i=1,2,3$) by Young's inequality, we get%
\begin{equation*}
\int\limits_{\Omega _{1}}\left\vert c_{0}\left( x\right) \right\vert
\left\vert u\right\vert ^{\alpha \left( x\right) }dx\leq \epsilon
_{1}\int\limits_{\Omega _{1}}\left( \frac{^{\alpha \left( x\right) }}{p_{0}}%
\right) \left\vert u\right\vert ^{p_{0}}dx+\int\limits_{\Omega _{1}}\left( 
\frac{1}{\epsilon _{1}}\right) ^{\frac{^{\alpha \left( x\right) }}{%
p_{0}-\alpha \left( x\right) }}\left( \frac{^{p_{0}-\alpha \left( x\right) }%
}{p_{0}}\right) \left\vert c_{0}\left( x\right) \right\vert ^{\frac{^{p_{0}}%
}{p_{0}-\alpha \left( x\right) }}dx
\end{equation*}%
\begin{equation*}
\leq \epsilon _{1}\int\limits_{\Omega _{1}}\left\vert u\right\vert
^{p_{0}}dx+\left( \frac{1}{\epsilon _{1}}\right) ^{\frac{p_{0}}{\eta }%
}\int\limits_{\Omega _{1}}\left\vert c_{0}\left( x\right) \right\vert ^{%
\frac{^{p_{0}}}{p_{0}-\alpha \left( x\right) }}dx.
\end{equation*}%
Similarly, by using \textbf{(ii) }and Young's inequality, we have the
following estimate for the fourth integral,%
\begin{equation*}
\int\limits_{\Omega _{2}}\left\vert c_{2}\left( x\right) \right\vert
\left\vert u\right\vert ^{\alpha _{1}\left( x\right) }dx\leq \epsilon
_{2}\int\limits_{\Omega _{2}}\left( \frac{^{\alpha _{1}\left( x\right) }}{%
p_{0}}\right) \left\vert u\right\vert ^{p_{0}}dx+\int\limits_{\Omega
_{2}}\left( \frac{1}{\epsilon _{2}}\right) ^{\frac{\alpha _{1}\left(
x\right) }{p_{0}-\alpha _{1}\left( x\right) }}\left( \frac{^{p_{0}-\alpha
_{1}\left( x\right) }}{p_{0}}\right) \left\vert c_{2}\left( x\right)
\right\vert ^{\frac{^{p_{0}}}{p_{0}-\alpha _{1}\left( x\right) }}dx
\end{equation*}%
\begin{equation*}
\leq \epsilon _{2}\int\limits_{\Omega _{2}}\left\vert u\right\vert
^{p_{0}}dx+\left( \frac{1}{\epsilon _{2}}\right) ^{\frac{\alpha _{1}^{+}}{%
p_{0}-\alpha _{1}^{+}}}\int\limits_{\Omega _{2}}\left\vert c_{2}\left(
x\right) \right\vert ^{\frac{^{p_{0}}}{p_{0}-\alpha _{1}\left( x\right) }}dx
\end{equation*}%
and for the second one by using H\"{o}lder-Young inequality, we take%
\begin{equation*}
\int\limits_{\Omega _{1}}\left\vert c_{1}\left( x\right) \right\vert
\left\vert u\right\vert dx\leq \epsilon _{3}\int\limits_{\Omega
_{1}}\left\vert u\right\vert ^{p_{0}}dx+\left( \frac{1}{\epsilon _{3}}%
\right) ^{p_{0}q_{0}}\int\limits_{\Omega _{1}}\left\vert c_{1}\left(
x\right) \right\vert ^{q_{0}}dx
\end{equation*}%
If we use these inequalities and condition \textbf{(iii)} in (3.6), we obtain%
\begin{equation*}
\langle f\left( u\right) ,u\rangle \geq \left[ u\right] _{\mathring{S}%
_{1,\left( p_{0}-2\right) ,2}\left( \Omega \right) }^{p_{0}}-\epsilon
_{4}\left\Vert u\right\Vert _{L^{p_{0}}\left( \Omega _{1}\cup \Omega
_{2}\right) }^{p_{0}}+\bar{C}_{0}\int\limits_{\Omega _{3}}\left\vert
u\right\vert ^{^{\alpha \left( x\right) }}dx-
\end{equation*}%
\begin{equation*}
-C_{1}(\epsilon _{1})-C_{2}(\epsilon _{2})-C_{3}(\epsilon _{3})-\left\Vert
c_{3}\right\Vert _{L^{1}\left( \Omega _{2}\right) }-\left\Vert
c_{5}\right\Vert _{L^{1}\left( \Omega _{3}\right) }
\end{equation*}%
estimating the first and second terms on the right-side of last inequality
by using Theorem 2.6, we obtain 
\begin{equation*}
\geq \tilde{C}\left[ u\right] _{\mathring{S}_{1,\left( p_{0}-2\right)
q_{0},q_{0}}\left( \Omega \right) }^{p_{0}}-\epsilon _{4}C_{4}\left[ u\right]
_{\mathring{S}_{1,\left( p_{0}-2\right) q_{0},q_{0}}\left( \Omega \right)
}^{p_{0}}+\bar{C}_{0}\int\limits_{\Omega _{3}}\left\vert u\right\vert
^{^{\alpha \left( x\right) }}dx-K
\end{equation*}%
\begin{equation*}
\geq C_{5}\left[ u\right] _{\mathring{S}_{1,\left( p_{0}-2\right)
q_{0},q_{0}}\left( \Omega \right) }^{p_{0}}+\bar{C}_{0}\int\limits_{\Omega
_{3}}\left\vert u\right\vert ^{^{\alpha \left( x\right) }}dx-K
\end{equation*}%
here, $K\equiv K\left( C_{1}(\epsilon _{1}),C_{2}(\epsilon
_{2}),C_{3}(\epsilon _{3}),\left\Vert c_{3}\right\Vert _{L^{1}\left( \Omega
_{2}\right) },\left\Vert c_{5}\right\Vert _{L^{1}\left( \Omega _{3}\right)
}\right) $, $\tilde{C}\equiv \tilde{C}\left( p_{0},mes\left( \Omega \right)
\right) >0$, $C_{5}\equiv C_{5}\left( p_{0},mes\left( \Omega \right) \right)
>0$, $C_{1}\equiv C_{1}\left( \sigma _{\beta }\left( c_{0}\right) ,\epsilon
_{1},p_{0}\right) >0$, $C_{2}\equiv C_{2}\left( \sigma _{\frac{^{p_{0}}}{%
p_{0}-\alpha _{1}\left( x\right) }}\left( c_{2}\right) ,\epsilon
_{2},p_{0},\alpha _{1}^{+}\right) >0$, $C_{3}\equiv C_{3}\left( \sigma
_{\beta _{1}}\left( c_{1}\right) ,\epsilon _{3},p_{0}\right) >0$ are
constants.

From last inequality we get that 
\begin{equation}
\langle f\left( u\right) ,u\rangle \geq C_{5}\left[ u\right] _{\mathring{S}%
_{1,\left( p_{0}-2\right) q_{0},q_{0}}\left( \Omega \right) }^{p_{0}}+\bar{C}%
_{0}\int\limits_{\Omega _{3}}\left\vert u\right\vert ^{^{\alpha \left(
x\right) }}dx-K  \tag{3.7}
\end{equation}%
If we take into account the following inequalities in (3.7) 
\begin{equation*}
\int\limits_{\Omega _{3}}\left\vert u\right\vert ^{^{\alpha \left( x\right)
}}dx\geq \left\Vert u\right\Vert _{L^{\alpha \left( x\right) }\left( \Omega
_{3}\right) }^{p_{0}}-1
\end{equation*}%
and%
\begin{equation*}
\left[ u\right] _{\mathring{S}_{1,\left( p_{0}-2\right) q_{0},q_{0}}\left(
\Omega _{2}\right) }^{p_{0}}\geq C_{6}\left\Vert u\right\Vert _{L^{\alpha
\left( x\right) }\left( \Omega _{2}\right) }^{p_{0}}\text{, }
\end{equation*}%
\begin{equation*}
\left[ u\right] _{\mathring{S}_{1,\left( p_{0}-2\right) q_{0},q_{0}}\left(
\Omega _{1}\right) }^{p_{0}}\geq C_{7}\left\Vert u\right\Vert _{L^{\alpha
\left( x\right) }\left( \Omega _{1}\right) }^{p_{0}}
\end{equation*}%
where $C_{6}\equiv C_{6}\left( p_{0},mes\left( \Omega \right) ,n\right) >0$, 
$C_{7}\equiv C_{7}\left( p_{0},mes\left( \Omega \right) \right) >0$ (comes
from Theorem 2.6 and Lemma 2.1), we obtain that 
\begin{equation*}
\langle f\left( u\right) ,u\rangle \geq C_{8}\left( \left[ u\right] _{%
\mathring{S}_{1,\left( p_{0}-2\right) q_{0},q_{0}}\left( \Omega \right)
}^{p_{0}}+\left\Vert u\right\Vert _{L^{\alpha \left( x\right) }\left( \Omega
\right) }^{p_{0}}\right) -\tilde{K}
\end{equation*}%
So the proof is completed.
\end{proof}

\begin{lemma}
Under the conditions of Theorem 3.2, the mapping $f$ defined by (3.4) is
bounded from $\mathring{S}_{1,\left( p_{0}-2\right) q_{0},q_{0}}\left(
\Omega \right) \cap L^{\alpha (x)}\left( \Omega \right) $ to $%
W^{-1,q_{0}}\left( \Omega \right) +L^{\alpha ^{\ast }(x)}\left( \Omega
\right) $.
\end{lemma}

\begin{proof}
Firstly we define the mappings%
\begin{equation*}
f_{1}\left( u\right) \equiv \sum_{i=1}^{n}-D_{i}\left( \left\vert
u\right\vert ^{p_{0}-2}D_{i}u\right)
\end{equation*}%
\begin{equation*}
f_{2}\left( u\right) \equiv c\left( x,u\right) .
\end{equation*}%
We need to show that, these mappings are bounded from $\mathring{S}%
_{1,\left( p_{0}-2\right) q_{0},q_{0}}\left( \Omega \right) \cap L^{\alpha
(x)}\left( \Omega \right) $ to $W^{-1,q_{0}}\left( \Omega \right) +L^{\alpha
^{\ast }(x)}\left( \Omega \right) .$

Let's show that $f_{1}$ is bounded: For $u\in \mathring{S}_{1,\left(
p_{0}-2\right) q_{0},q_{0}}\left( \Omega \right) $ and $v\in
W_{0}^{1,p_{0}}\left( \Omega \right) $%
\begin{equation*}
\left\vert \langle f_{1}\left( u\right) ,v\rangle \right\vert \leq
\sum_{i=1}^{n}\left( \int\limits_{\Omega }\left\vert u\right\vert
^{p_{0}-2}\left\vert D_{i}u\right\vert \left\vert D_{i}v\right\vert dx\right)
\end{equation*}%
Using H\"{o}lder's inequality we get%
\begin{equation*}
\leq \left[ \sum_{i=1}^{n}\left( \int\limits_{\Omega }\left\vert
u\right\vert ^{(p_{0}-2)q_{0}}\left\vert D_{i}u\right\vert ^{q_{0}}dx\right) %
\right] ^{\frac{1}{q_{0}}}\left[ \sum_{i=1}^{n}\left( \int\limits_{\Omega
}\left\vert D_{i}v\right\vert ^{p_{0}}dx\right) \right] ^{\frac{1}{p_{0}}}
\end{equation*}%
\begin{equation*}
=\left[ u\right] _{\mathring{S}_{1,\left( p_{0}-2\right) q_{0},q_{0}}\left(
\Omega \right) }^{p_{0}-1}\left\Vert v\right\Vert _{W_{0}^{1,p_{0}}\left(
\Omega \right) }
\end{equation*}%
From last inequality we take the boundness of $f_{1}$.

Similarly by using (3.1) and Theorem 2.6, $\forall $ $u\in \mathring{S}%
_{1,\left( p_{0}-2\right) q_{0},q_{0}}\left( \Omega \right) \cap L^{\alpha
(x)}\left( \Omega \right) $, we have the following estimate%
\begin{equation*}
\sigma _{\alpha ^{\ast }}\left( f_{2}\left( u\right) \right) =\sigma
_{\alpha ^{\ast }}\left( c\left( x,u\right) \right) \leq C_{9}\left( \sigma
_{\alpha }\left( u\right) +\left[ u\right] _{\mathring{S}_{1,\left(
p_{0}-2\right) q_{0},q_{0}}\left( \Omega \right) }^{\tilde{p}}\right) +C_{10}
\end{equation*}%
here $C_{9}=C_{9}\left( \alpha ^{+},\alpha ^{-}\right) >0,$ $%
C_{10}=C_{10}\left( \sigma _{\beta }\left( c_{0}\right) ,\sigma _{\beta
_{1}}\left( c_{1}\right) ,mes\left( \Omega \right) \right) >0$ are
constants. Thus we obtain that $f_{2}:$ $\mathring{S}_{1,\left(
p_{0}-2\right) q_{0},q_{0}}\left( \Omega \right) \cap L^{\alpha (x)}\left(
\Omega \right) $ $\rightarrow $ $L^{\alpha ^{\ast }(x)}\left( \Omega \right) 
$ is bounded.
\end{proof}

\begin{lemma}
Under the conditions of Theorem 3.2, the mapping $f$ defined by (3.4) is
weak compact from $\mathring{S}_{1,\left( p_{0}-2\right) q_{0},q_{0}}\left(
\Omega \right) \cap L^{\alpha (x)}\left( \Omega \right) $ to $%
W^{-1,q_{0}}\left( \Omega \right) +L^{\alpha ^{\ast }(x)}\left( \Omega
\right) .$
\end{lemma}

\begin{proof}
Firstly we want to see the weak compactness of $f_{1}$. For $\left\{
u_{m}\right\} _{m=1}^{\infty }\subset \mathring{S}_{1,\left( p_{0}-2\right)
q_{0},q_{0}}\left( \Omega \right) \cap L^{\alpha (x)}\left( \Omega \right) $
bounded and $u_{m}\overset{\text{ }S_{0}}{\rightharpoonup }u_{0}$ it is
sufficent to show a subsequence of $\left\{ u_{m_{j}}\right\} _{m=1}^{\infty
}\subset \left\{ u_{m}\right\} _{m=1}^{\infty }$ which satisfies $%
f_{1}\left( u_{m_{j}}\right) $ $\overset{\text{ }W^{-1,q_{0}}\left( \Omega
\right) }{\rightharpoonup }f_{1}\left( u_{0}\right) $

Since we have one-to-one correspondence between the classes (Theorem 2.5)%
\begin{equation*}
\mathring{S}_{1,\left( p_{0}-2\right) q_{0},q_{0}}\left( \Omega \right) 
\underset{\varphi ^{-1}}{\overset{\varphi }{\longleftrightarrow }}%
W_{0}^{1,q_{0}}\left( \Omega \right)
\end{equation*}%
with the homeomorfizm%
\begin{equation*}
\varphi \left( \tau \right) \equiv \left\vert \tau \right\vert
^{p_{0}-2}\tau ,\text{ }\varphi ^{-1}\left( \tau \right) \equiv \left\vert
\tau \right\vert ^{-\frac{p_{0}-2}{p_{0}-1}}\tau
\end{equation*}%
for $\forall m\geq 1$%
\begin{equation*}
\left\vert u_{m}\right\vert ^{p_{0}-2}u_{m}\in W_{0}^{1,q_{0}}\left( \Omega
\right)
\end{equation*}%
and since $W_{0}^{1,q_{0}}\left( \Omega \right) $ is reflexive space, there
exist a subsequence $\left\{ u_{m_{j}}\right\} _{m=1}^{\infty }\subset
\left\{ u_{m}\right\} _{m=1}^{\infty }$ such that 
\begin{equation*}
\left\vert u_{m_{j}}\right\vert ^{p_{0}-2}u_{m_{j}}\overset{\text{ }%
W_{0}^{1,q_{0}}\left( \Omega \right) }{\rightharpoonup }\xi \text{ }
\end{equation*}%
Now we will show that $\xi =\left\vert u_{0}\right\vert ^{p_{0}-2}u_{0}.$

According to compact embedding, $W_{0}^{1,q_{0}}\left( \Omega \right)
\hookrightarrow L^{q_{0}}\left( \Omega \right) $ 
\begin{equation*}
\exists \left\{ u_{m_{j_{k}}}\right\} _{m=1}^{\infty }\subset \left\{
u_{m_{j}}\right\} _{m=1}^{\infty }\text{, }\left\vert
u_{m_{j_{k}}}\right\vert ^{p_{0}-2}u_{m_{j_{k}}}\overset{\text{ }%
L^{q_{0}}\left( \Omega \right) }{\rightarrow }\xi \text{ }
\end{equation*}%
Since $\varphi ^{-1}:L^{q_{0}}\left( \Omega \right) \longrightarrow
L^{p_{0}}\left( \Omega \right) $ continous then%
\begin{equation*}
u_{m_{j_{k}}}\overset{\text{ }L^{p_{0}}\left( \Omega \right) }{\rightarrow }%
\varphi ^{-1}\left( \xi \right) \text{ }
\end{equation*}%
hence we have%
\begin{equation*}
u_{m_{j_{k}}}\underset{a.e}{\overset{\Omega }{\rightarrow }}\varphi
^{-1}\left( \xi \right) \text{ }
\end{equation*}%
So we obtain $\varphi ^{-1}\left( \xi \right) =u_{0\text{ }}$ or $\xi
=\left\vert u_{0}\right\vert ^{p_{0}-2}u_{0}.$

From this, we conclude that for $\forall v\in W_{0}^{1,p_{0}}\left( \Omega
\right) $ 
\begin{equation*}
\langle f_{1}\left( u_{m_{j_{k}}}\right) ,v\rangle =\sum_{i=1}^{n}\langle
-D_{i}\left( \left\vert u_{m_{j_{k}}}\right\vert
^{p_{0}-2}D_{i}u_{m_{j_{k}}}\right) ,v\rangle
\end{equation*}%
\begin{equation*}
\underset{m_{j}\nearrow \infty }{\longrightarrow }\sum_{i=1}^{n}\langle
-D_{i}\left( \left\vert u_{0}\right\vert ^{p_{0}-2}D_{i}u_{0}\right)
,v\rangle =\langle f_{1}\left( u_{0}\right) ,v\rangle
\end{equation*}%
hence, the result is obtained.

Now we shall show the weak compactness of $f_{2}$. Since%
\begin{equation*}
c:\mathring{S}_{1,\left( p_{0}-2\right) q_{0},q_{0}}\left( \Omega \right)
\cap L^{\alpha (x)}\left( \Omega \right) \rightarrow L^{\alpha ^{\ast
}(x)}\left( \Omega \right)
\end{equation*}%
is bounded (Lemma 3.4), then $\forall m\geq 1,$ $f_{2}\left( u_{m}\right)
\equiv \left\{ c\left( x,u_{m}\right) \right\} _{m=1}^{\infty }\subset
L^{\alpha ^{\ast }(x)}\left( \Omega \right) .$ Also $L^{\alpha ^{\ast
}(x)}\left( \Omega \right) $ ($1<\left( \alpha ^{\ast }\right) ^{-}<\infty $%
) is reflexive space thus $\left\{ u_{m}\right\} _{m=1}^{\infty }$ has a
subsequence $\left\{ u_{m_{j}}\right\} _{m=1}^{\infty }$ such that%
\begin{equation*}
c\left( x,u_{m_{j}}\right) \overset{\text{ }L^{\alpha ^{\ast }(x)}\left(
\Omega \right) }{\rightharpoonup }\psi
\end{equation*}%
Since we have the compact embedding, $\mathring{S}_{1,\left( p_{0}-2\right)
q_{0},q_{0}}\left( \Omega \right) \hookrightarrow L^{p_{0}}\left( \Omega
\right) $%
\begin{equation*}
\exists \left\{ u_{m_{j_{k}}}\right\} _{m=1}^{\infty }\subset \left\{
u_{m_{j}}\right\} _{m=1}^{\infty }\text{, }u_{m_{j_{k}}}\overset{%
L^{p_{0}}\left( \Omega \right) }{\rightarrow }u_{0}\text{ }
\end{equation*}%
thus%
\begin{equation*}
u_{m_{j_{k}}}\underset{a.e}{\overset{\Omega }{\rightarrow }}u_{0}\text{ }
\end{equation*}%
and using the continuity of $c\left( x,.\right) $ for almost $x\in \Omega $,
we get%
\begin{equation*}
c(x,u_{m_{j_{k}}})\underset{a.e}{\overset{\Omega }{\rightarrow }}c\left(
x,u_{0}\right) \text{ }
\end{equation*}%
so, we arrive at $\psi =c\left( x,u_{0}\right) $ i.e. $f_{2}(u_{m_{j_{k}}})%
\overset{W^{-1,q_{0}}\left( \Omega \right) +L^{\alpha ^{\ast }(x)}\left(
\Omega \right) }{\rightharpoonup }f_{2}\left( u_{0}\right) $.
\end{proof}

\bigskip

Now we give the proof of main theorem of this section.

\begin{proof}
(\textbf{of Theorem 3.2}) Since $g\equiv Id,$ so is linear bounded map and\
fulfills the conditions of (1). Also from Lemma 3.3-Lemma 3.5, it follows
that the mappings $f$ and $g$ satisfy all the conditions of Theorem 2.9. If
we apply Theorem 2.9 to problem (1.2), we obtain that $\forall h\in
W^{-1,q_{0}}\left( \Omega \right) +L^{\alpha ^{\ast }\left( x\right) }\left(
\Omega \right) $ the equation%
\begin{equation*}
\sum_{i=1}^{n}-\int\limits_{\Omega }\left[ D_{i}\left( \left\vert
u\right\vert ^{p_{0}-2}D_{i}u\right) +c\left( x,u\right) \right]
wdx=\int\limits_{\Omega }h\left( x\right) wdx,\text{ }w\in
W_{0}^{1,p_{0}}\left( \Omega \right) \cap L^{\alpha (x)}\left( \Omega \right)
\end{equation*}%
has a solution in $\mathring{S}_{1,\left( p_{0}-2\right) q_{0},q_{0}}\left(
\Omega \right) \cap L^{\alpha (x)}\left( \Omega \right) $.
\end{proof}

\bigskip

It can be easily seen from the proof of Theorem 3.2, the results which is
given below are valid for special conditions of $\alpha .$

\begin{corollary}
Let \textbf{(i)} holds. If $\ 1<\alpha ^{-}\leq \alpha \left( x\right) \leq
\alpha ^{+}<p_{0}$ i.e. $\Omega \equiv \Omega _{1}$ and $c_{0}\in L^{\beta
_{2}\left( x\right) }\left( \Omega \right) $, $c_{1}\in L^{\alpha ^{\ast
}\left( x\right) }\left( \Omega \right) ~$where $\beta _{2}\left( x\right)
\equiv \frac{p_{0}\alpha ^{\ast }\left( x\right) }{p_{0}-\alpha \left(
x\right) }$ then $\forall h\in W^{-1,q_{0}}\left( \Omega \right) $ problem
(1.2) has a generalized solution in the space $\mathring{S}_{1,\left(
p_{0}-2\right) q_{0},q_{0}}\left( \Omega \right) $.
\end{corollary}

\begin{corollary}
Let \textbf{(i)}, \textbf{(ii)} hold. If $1<\alpha ^{-}\leq \alpha \left(
x\right) \leq \alpha ^{+}<\tilde{p}$ i.e. $\Omega \equiv \Omega _{1}\cup
\Omega _{2}$ and $c_{2}\in L^{\frac{p_{0}}{p_{0}-\alpha _{1}\left( x\right) }%
}\left( \Omega _{2}\right) ,$ $c_{3}\in L^{1}\left( \Omega _{2}\right) $, $%
c_{0}\in L^{\beta _{3}\left( x\right) }\left( \Omega \right) $, $c_{1}\in
L^{\beta _{4}\left( x\right) }\left( \Omega \right) $ where$\ \beta
_{3}\left( x\right) \equiv \left\{ 
\begin{array}{l}
\frac{p_{0}\alpha ^{\ast }\left( x\right) }{p_{0}-\alpha \left( x\right) }%
\text{ if\ }x\in \Omega _{1} \\ 
\frac{\tilde{p}\alpha ^{\ast }\left( x\right) }{\tilde{p}-\alpha \left(
x\right) }\text{ if\ }x\in \Omega _{2}%
\end{array}%
\right. $ and $\beta _{4}\left( x\right) \equiv \left\{ 
\begin{array}{l}
\alpha ^{\ast }\left( x\right) \text{ if\ }x\in \Omega _{1} \\ 
q_{0}\text{ \ \ \ \ \ if\ }x\in \Omega _{2}%
\end{array}%
\right. $ then $\forall h\in W^{-1,q_{0}}\left( \Omega \right) $ problem
(1.2) has a generalized solution in the space $\mathring{S}_{1,\left(
p_{0}-2\right) q_{0},q_{0}}\left( \Omega \right) $.
\end{corollary}

\section{\protect\bigskip {\protect\large Existence Results for Main Problem
(1.1)}}

\subsection{\protect\normalsize Preliminary results}

In this subsection, we prove some necessary results. Througout this section,
we take $\Omega \subset 
\mathbb{R}
^{n}\left( n\geq 2\right) $ be a bounded domain with sufficently smooth
boundary.

\begin{lemma}
Assume that $\zeta :\Omega \longrightarrow \left[ 1,\infty \right) $ is
measurable function that satisfy $1\leq \zeta ^{-}\leq \zeta \left( x\right)
\leq $ $\zeta ^{+}<\infty $ also $\beta >1,$ $\epsilon >0.$ Then for every $%
u\in L^{\zeta \left( x\right) +\epsilon }\left( \Omega \right) $%
\begin{equation*}
\int\limits_{\Omega }\left\vert u\right\vert ^{\zeta \left( x\right)
}\left\vert \ln \left\vert u\right\vert \right\vert ^{\beta }dx\leq
M_{1}\int\limits_{\Omega }\left\vert u\right\vert ^{\zeta \left( x\right)
+\epsilon }dx+M_{2}
\end{equation*}%
is satisfied. Here $M_{1}\equiv M_{1}\left( \epsilon ,\beta \right) >0$ and $%
M_{2}\equiv M_{2}\left( \epsilon ,\beta ,mes\left( \Omega \right) \right) >0$
are constants.
\end{lemma}

\begin{proof}
For given $\epsilon >0$ one can easily see that by calculus there exist $%
M_{0}=M_{0}\left( \epsilon \right) >0$ such that 
\begin{equation*}
\ln \left\vert t\right\vert \leq M_{0}\left( \epsilon \right) \left\vert
t\right\vert ^{\epsilon },\text{ }t\in 
\mathbb{R}
-\left\{ 0\right\}
\end{equation*}%
is hold. Hence on the set $\left\{ x\in \Omega :\left\vert u\left( x\right)
\right\vert \geq 1\text{ }\right\} $ the inequality $\left\vert u\right\vert
^{\zeta \left( x\right) }\left\vert \ln \left\vert u\right\vert \right\vert
^{\beta }\leq M_{0}\left( \epsilon ,\beta \right) \left\vert u\right\vert
^{\zeta \left( x\right) +\epsilon }$ is satisfied. On the other hand, since $%
\underset{x\rightarrow 0^{+}}{\lim }t^{\epsilon }\left\vert \ln t\right\vert
^{\beta }=0$ and for every fixed $x_{0}\in \Omega $, $\underset{x\rightarrow
0^{+}}{\lim }\frac{\left\vert t\right\vert ^{\zeta \left( x_{0}\right)
}\left\vert \ln \left\vert t\right\vert \right\vert ^{\beta }}{t^{\zeta
\left( x_{0}\right) +\epsilon }+1}=0$ we have the inequality $\left\vert
u\right\vert ^{\zeta \left( x\right) -1}\left\vert u\right\vert \left\vert
\ln \left\vert u\right\vert \right\vert ^{\beta }$ $\leq \tilde{M}_{0}\left(
\left\vert u\right\vert ^{\zeta \left( x\right) +\epsilon }+1\right) $ on
the set $\left\{ x\in \Omega :\left\vert u\left( x\right) \right\vert <1%
\text{ }\right\} $ for some $\tilde{M}_{0}=\tilde{M}_{0}\left( \epsilon
,\beta \right) >0$. So the proof is completed from the combination of these
inequalities.
\end{proof}

\bigskip

Let $\rho :$ $\Omega \longrightarrow 
\mathbb{R}
$, $2\leq \rho ^{-}\leq \rho \left( x\right) \leq \rho ^{+}<\infty $, $\rho
\in C^{1}\left( \bar{\Omega}\right) $. $m$ be a number which satisfies $\rho
\left( x\right) \geq $ $m\geq 2$ a.e. $x\in \Omega $, $m_{1}\equiv \frac{m}{%
m-1}$, $\psi (x)\equiv \tfrac{\rho \left( x\right) -m}{m-1}$, $x\in \Omega $
and $\varphi :\Omega \longrightarrow 
\mathbb{R}
$ is measurable function satisfy $1\leq $ $\varphi ^{-}\leq \varphi \left(
x\right) \leq \varphi ^{+}<\infty $\textbf{. }

\bigskip

\noindent Now we introduce the following class of functions for $u:\Omega
\longrightarrow 
\mathbb{R}
$: 
\begin{equation*}
\tilde{T}_{0}\equiv \left\{ u\in L^{1}\left( \Omega \right) \mid
\sum\limits_{i=1}^{n}\left\Vert \left\vert u\right\vert ^{\psi \left(
x\right) }D_{i}u\right\Vert _{L^{m}\left( \Omega \right) }+\left\Vert
u\right\Vert _{L^{\frac{nm_{1}\left( \rho \left( x\right) -1\right) }{n-m_{1}%
}}\left( \Omega \right) }<\infty \right\} \cap
\end{equation*}%
\begin{equation*}
\cap L^{\varphi \left( x\right) +\psi \left( x\right) }\left( \Omega \right)
\cap \left\{ u\mid _{\partial \Omega }\equiv 0\right\}
\end{equation*}%
and%
\begin{equation*}
T_{0}\equiv \left\{ u\in L^{1}\left( \Omega \right) \mid
\sum\limits_{i=1}^{n}\left\Vert \left\vert u\right\vert ^{\rho \left(
x\right) -2}D_{i}u\right\Vert _{L^{m_{1}}\left( \Omega \right) }+\left\Vert
u\right\Vert _{L^{\frac{nm_{1}\left( \rho \left( x\right) -1\right) }{n-m_{1}%
}}\left( \Omega \right) }<\infty \right\} \cap
\end{equation*}%
\begin{equation*}
\cap L^{\varphi \left( x\right) +\psi \left( x\right) }\left( \Omega \right)
\cap \left\{ u\mid _{\partial \Omega }\equiv 0\right\}
\end{equation*}%
\footnote{%
In general, there might not be an embedding between $L^{\varphi \left(
x\right) +\psi (x)}\left( \Omega \right) $ and $L^{\frac{nm_{1}\left( \rho
\left( x\right) -1\right) }{n-m_{1}}}\left( \Omega \right) $.}

\bigskip

Following lemma gives the relation of these classes with Sobolev and
Generalized Lebesgue spaces;

\begin{lemma}
\bigskip Let the functions $\rho $, $\varphi $, $\psi $ and number $m$ be
defined such as above, then the following statements hold: 
\begin{equation*}
\text{(a) }\phi :\Omega \times 
\mathbb{R}
\longrightarrow 
\mathbb{R}
\text{, }\phi \left( x,\tau \right) \equiv \left\vert \tau \right\vert
^{\rho \left( x\right) -2}\tau \text{ }
\end{equation*}%
is a bijection between the spaces $T_{0}$ and $W_{0}^{1,m_{1}}\left( \Omega
\right) \cap $ $L^{\frac{\varphi \left( x\right) +\psi \left( x\right) }{%
\left( m-1\right) \left( \psi \left( x\right) +1\right) }}\left( \Omega
\right) $. 
\begin{equation*}
\text{(b) }\tilde{\phi}:\Omega \times 
\mathbb{R}
\longrightarrow 
\mathbb{R}
\text{, }\tilde{\phi}\left( x,\tau \right) \equiv \left\vert \tau
\right\vert ^{\psi \left( x\right) }\tau
\end{equation*}%
is a bijection between $\tilde{T}_{0}$ and $W_{0}^{1,m}\left( \Omega \right)
\cap L^{\frac{\varphi \left( x\right) +\psi \left( x\right) }{\psi \left(
x\right) +1}}\left( \Omega \right) $.
\end{lemma}

\begin{proof}
Since the proofs of (a) and (b) are similar, we only prove (a).

First let us to show that for $u\in T_{0}$, $v\equiv \left\vert u\right\vert
^{\rho \left( x\right) -2}u\equiv \phi \left( x,u\right) \in
W_{0}^{1,m_{1}}\left( \Omega \right) \cap $ $L^{\frac{\varphi \left(
x\right) +\psi \left( x\right) }{\left( m-1\right) \left( \psi \left(
x\right) +1\right) }}\left( \Omega \right) $. Since by direct calculations%
\begin{equation*}
\sigma _{\frac{\varphi +\psi }{\left( m-1\right) \left( \psi +1\right) }%
}\left( v\right) =\sigma _{\varphi +\psi }\left( u\right)
\end{equation*}%
from this equality, we take $v\in $ $L^{\frac{\varphi \left( x\right) +\psi
\left( x\right) }{\left( m-1\right) \left( \psi \left( x\right) +1\right) }%
}\left( \Omega \right) $.

On the other hand for $\forall i=\overline{1..n}$%
\begin{equation*}
\left\Vert D_{i}v\right\Vert _{m_{1}}^{m_{1}}=\int\limits_{\Omega
}\left\vert D_{i}\left( \left\vert u\right\vert ^{\rho \left( x\right)
-2}u\right) \right\vert ^{m_{1}}dx=
\end{equation*}%
\begin{equation*}
\int\limits_{\Omega }\left\vert \left( \rho \left( x\right) -1\right)
\left\vert u\right\vert ^{\rho \left( x\right) -2}D_{i}u+\left( D_{i}\rho
\right) \left\vert u\right\vert ^{\rho \left( x\right) -2}u\ln \left\vert
u\right\vert \right\vert ^{m_{1}}dx\leq
\end{equation*}%
\begin{equation*}
\leq C_{0}\int\limits_{\Omega }\left\vert u\right\vert ^{m_{1}\left( \rho
\left( x\right) -2\right) }\left\vert D_{i}u\right\vert
^{m_{1}}dx+C_{1}\int\limits_{\Omega }\left\vert u\right\vert ^{m_{1}\left(
\rho \left( x\right) -1\right) }\left\vert \ln \left\vert u\right\vert
\right\vert ^{m_{1}}dx
\end{equation*}%
here $C_{0}=C_{0}\left( m_{1},\left\Vert \rho \right\Vert _{C\left( \Omega
\right) }\right) $, $C_{1}=C_{1}\left( m_{1},\left\Vert \rho \right\Vert
_{C^{1}\left( \bar{\Omega}\right) }\right) >0$ are constants. As for
sufficently small $\varepsilon >0$, $m_{1}\left( \rho \left( x\right)
-1\right) +\varepsilon \leq \frac{nm_{1}\left( \rho \left( x\right)
-1\right) }{n-m_{1}}$ holds, applying Lemma 4.1 to second integral we obtain%
\begin{equation*}
\leq C_{0}\left\Vert \left\vert u\right\vert ^{\rho \left( x\right)
-2}D_{i}u\right\Vert _{m_{1}}^{m_{1}}+C_{2}\int\limits_{\Omega }\left\vert
u\right\vert ^{m_{1}\left( \rho \left( x\right) -1\right) +\varepsilon
}dx+C_{3}
\end{equation*}%
Applying Young's inequality to second integral, we take that 
\begin{equation*}
\leq C_{0}\left\Vert \left\vert u\right\vert ^{\rho \left( x\right)
-2}D_{i}u\right\Vert _{m_{1}}^{m_{1}}+C_{2}\sigma _{\frac{nm_{1}\left( \rho
\left( x\right) -1\right) }{n-m_{1}}}\left( u\right) +\tilde{C}_{3}
\end{equation*}%
where $C_{2}=C_{2}\left( m_{1},\left\Vert \rho \right\Vert _{C^{1}\left( 
\bar{\Omega}\right) },\varepsilon \right) >0$ and $\tilde{C}_{3}=\tilde{C}%
_{3}\left( m_{1},\left\Vert \rho \right\Vert _{C^{1}\left( \bar{\Omega}%
\right) },mes\left( \Omega \right) \right) >0$. Hence from last inequality,
we get $v\in W_{0}^{1,m_{1}}\left( \Omega \right) $.

Conversely for all $v\in W_{0}^{1,m_{1}}\left( \Omega \right) \cap L^{\frac{%
\varphi \left( x\right) +\psi \left( x\right) }{\left( m-1\right) \left(
\psi \left( x\right) +1\right) }}\left( \Omega \right) $ let us to show that 
$w\equiv \left\vert v\right\vert ^{-\frac{\rho \left( x\right) -2}{\rho
\left( x\right) -1}}v\equiv \phi ^{-1}\left( x,v\right) \in T_{0}$. As%
\begin{equation*}
\sigma _{\varphi +\psi }\left( w\right) =\sigma _{\frac{\varphi +\psi }{%
\left( m-1\right) \left( \psi +1\right) }}\left( v\right)
\end{equation*}%
according to this equality, we take $w\in L^{\varphi \left( x\right) +\psi
\left( x\right) }\left( \Omega \right) $. Furthermore from definition of $%
T_{0}$ and the Luxemburg norm, we have 
\begin{equation*}
\left\Vert \left\vert w\right\vert ^{\rho \left( x\right)
-2}D_{i}w\right\Vert _{m_{1}}^{m_{1}}+\sigma _{_{\frac{nm_{1}\left( \rho
\left( x\right) -1\right) }{n-m_{1}}}}\left( w\right) =
\end{equation*}%
\begin{equation*}
=\int\limits_{\Omega }\left\vert v\right\vert ^{\frac{m\left( \rho \left(
x\right) -2\right) }{\rho \left( x\right) -1}}\left\vert D_{i}\left(
\left\vert v\right\vert ^{-\frac{\rho \left( x\right) -2}{\rho \left(
x\right) -1}}v\right) \right\vert ^{m_{1}}dx+\int\limits_{\Omega }\left\vert
v\right\vert ^{\frac{nm_{1}}{n-m_{1}}}dx=
\end{equation*}%
\begin{equation*}
=\int\limits_{\Omega }\left\vert v\right\vert ^{\frac{m\left( \rho \left(
x\right) -2\right) }{\rho \left( x\right) -1}}\left\vert \left( \tfrac{1}{%
\rho \left( x\right) -1}\right) \left\vert v\right\vert ^{-\frac{\rho \left(
x\right) -2}{\rho \left( x\right) -1}}D_{i}v+\left( \tfrac{-D_{i}\left( \rho
\right) }{\left( \rho \left( x\right) -1\right) ^{2}}\right) \left\vert
v\right\vert ^{-\frac{\rho \left( x\right) -2}{\rho \left( x\right) -1}}v\ln
\left\vert v\right\vert \right\vert ^{m_{1}}+
\end{equation*}%
\begin{equation*}
+\int\limits_{\Omega }\left\vert v\right\vert ^{\frac{nm_{1}}{n-m_{1}}%
}dx\leq C_{4}\int\limits_{\Omega }\left\vert D_{i}v\right\vert
^{m_{1}}dx+C_{5}\int\limits_{\Omega }\left\vert v\right\vert
^{m_{1}}\left\vert \ln \left\vert v\right\vert \right\vert
^{m_{1}}dx+\int\limits_{\Omega }\left\vert v\right\vert ^{\frac{nm_{1}}{%
n-m_{1}}}dx
\end{equation*}%
here $C_{4}=C_{4}\left( m_{1},\left\Vert \rho \right\Vert _{C\left( \bar{%
\Omega}\right) }\right) >0$ and $C_{5}=C_{5}\left( m_{1},\left\Vert \rho
\right\Vert _{C^{1}\left( \bar{\Omega}\right) }\right) >0$.

Estimating the second integral in the right of last inequality with the help
of Lemma 4.1, we obtain%
\begin{equation*}
\leq C_{4}\int\limits_{\Omega }\left\vert D_{i}v\right\vert
^{m_{1}}dx+C_{6}\int\limits_{\Omega }\left\vert v\right\vert ^{\frac{nm_{1}}{%
n-m_{1}}}dx+C_{7}.
\end{equation*}%
Using the embedding $W_{0}^{1,m_{1}}\left( \Omega \right) \subset L^{\frac{%
nm_{1}}{n-m_{1}}}\left( \Omega \right) $ [1] in the last inequality, we
obtain $w\in T_{0}$.

To end the proof, observe that for every fixed $x_{0}\in \Omega $, $\phi _{%
\text{ }}\left( x_{0},\tau \right) \equiv \phi _{\text{ }}\left( \tau
\right) =\left\vert \tau \right\vert ^{\rho \left( x_{0}\right) -2}\tau $
and $\phi _{\text{ }}^{-1}\left( x_{0},t\right) \equiv \phi _{\text{ }%
}^{-1}\left( t\right) =\left\vert t\right\vert ^{-\frac{\rho \left(
x_{0}\right) -2}{\rho \left( x_{0}\right) -1}}t$ are strictly monotone
functions thus we take that $\phi $ is a bijection between $T_{0}$ and $%
W_{0}^{1,m_{1}}\left( \Omega \right) \cap $ $L^{\frac{\varphi \left(
x\right) +\psi \left( x\right) }{\left( m-1\right) \left( \psi \left(
x\right) +1\right) }}\left( \Omega \right) $.
\end{proof}

\bigskip

\begin{remark}
\noindent Under the conditions of Lemma 4.2, $T_{0}$ and $\tilde{T}_{0}$ are
metric spaces.\ On the class $T_{0}$, metric is defined as given below: $%
\forall u,$ $v\in T_{0}$%
\begin{equation*}
d_{T_{0}}\left( u,v\right) \equiv \left\Vert \phi _{\text{ }}\left( u\right)
-\phi _{\text{ }}\left( v\right) \right\Vert _{L^{\frac{\varphi \left(
x\right) +\psi \left( x\right) }{\left( m-1\right) \left( \psi \left(
x\right) +1\right) }}\left( \Omega \right) }+\left\Vert \phi _{\text{ }%
}\left( u\right) -\phi _{\text{ }}\left( v\right) \right\Vert
_{W_{0}^{1,m_{1}}\left( \Omega \right) }
\end{equation*}%
it easy to see that $d_{T_{0}}\left( .,.\right) :T_{0}\longrightarrow 
\mathbb{R}
$ fulfills the metric conditions and moreover $\phi $ and $\phi _{\text{ }%
}^{-1}$ are continous in the sense of topology defined on $T_{0}$ with this
metric. Hence we get that $\phi $ is a homeomorphism between $T_{0}$ and $%
W_{0}^{1,m_{1}}\left( \Omega \right) \cap $ $L^{\frac{\varphi \left(
x\right) +\psi \left( x\right) }{\left( m-1\right) \left( \psi \left(
x\right) +1\right) }}\left( \Omega \right) $. By the same way we can show
that $\tilde{\phi}$ is a homeomorphism between $\tilde{T}_{0}$ and $%
W_{0}^{1,m}\left( \Omega \right) \cap L^{\frac{\varphi \left( x\right) +\psi
\left( x\right) }{\psi \left( x\right) +1}}\left( \Omega \right) $.
\end{remark}

\subsection{\protect\bigskip {\protect\normalsize Solvability of Problem
(1.1)}}

In this section, we consider our main problem (1.1) and investigate the
existence of weak solution of that problem by the help of the results that
established in Subsection 4.1 and Theorem 3.2. {So, we study }%
\begin{equation}
\left\{ 
\begin{array}{l}
-\ \Delta \left( \left\vert u\right\vert ^{p(x)-2}u\right) +a\left(
x,u\right) =h\left( x\right) \\ 
\text{ \ \ }u\mid _{\partial \Omega }=0%
\end{array}%
\right.  \tag{1.1}
\end{equation}%
under the following conditions:

\textbf{(I)} $p:$\textit{\ }$\Omega \longrightarrow 
\mathbb{R}
$\textit{, }$2\leq p^{-}\leq p\left( x\right) \leq p^{+}<\infty $\textit{\
and }$p\in $\textit{\ }$C^{1}\left( \bar{\Omega}\right) $\textit{, }$%
a:\Omega \times 
\mathbb{R}
\longrightarrow 
\mathbb{R}
$\textit{\ is a Caratheodory function and for the measurable function }$\xi
:\Omega \longrightarrow 
\mathbb{R}
$\textit{\ satisfies }$1<\xi ^{-}\leq \xi \left( x\right) \leq \xi
^{+}<\infty $\textit{, the inequality}%
\begin{equation}
\left\vert a\left( x,\tau \right) \right\vert \leq a_{0}\left( x\right)
\left\vert \tau \right\vert ^{\xi \left( x\right) -1}+a_{1}\left( x\right) ,%
\text{ }\left( x,\tau \right) \in \Omega \times 
\mathbb{R}
,  \tag{4.1}
\end{equation}%
\textit{holds. Here }$a_{0},$\textit{\ }$a_{1}$\textit{\ are nonnegative,
measurable functions defined on }$\Omega $\textit{.}

\textit{Let }$\eta _{0}\in \left( 0,1\right) $\textit{\ is sufficently
small. We separate }$\Omega $\textit{\ to disjoint sets because of the same
reason for which is required problem (1.2).}%
\begin{equation}
\left. 
\begin{array}{l}
\Omega _{1}\equiv \left\{ x\in \Omega \text{ }|\text{ }1\leq \xi ^{-}\leq
\xi \left( x\right) \leq p\left( x\right) -\eta _{0}\right\} \\ 
\Omega _{2}\equiv \left\{ x\in \Omega \text{ }|\text{ }p\left( x\right)
-\eta _{0}<\xi \left( x\right) \leq \tilde{p}\left( x\right) \right\} \\ 
\Omega _{3}\equiv \left\{ x\in \Omega \text{ }|\text{ }\tilde{p}\left(
x\right) \leq \xi \left( x\right) \leq \xi ^{+}<\infty \right\}%
\end{array}%
\right.  \tag{*}
\end{equation}%
\textit{here }$\tilde{p}:\Omega \longrightarrow 
\mathbb{R}
$\textit{\ is a measurable function which satisfies }$2\leq p\left( x\right)
<\tilde{p}\left( x\right) $\textit{\ a.e. }$x\in \Omega $\textit{\ and will
be defined later.}

\textbf{(II)}\textit{\ There exist a measurable function }$\xi _{1}:\Omega
_{2}\longrightarrow 
\mathbb{R}
$\textit{\ which satisfy }$2\leq $\textit{\ }$\xi _{1}^{-}\leq \xi
_{1}\left( x\right) \leq \xi _{1}^{+}<p\left( x\right) $\textit{, such that
on }$\Omega _{2}\times 
\mathbb{R}
$\textit{, }$a\left( x,\tau \right) $\textit{\ fulfills the inequality}%
\begin{equation}
a\left( x,\tau \right) \tau \geq -a_{2}\left( x\right) \left\vert \tau
\right\vert ^{\xi _{1}\left( x\right) }-a_{3}\left( x\right) ,\text{ }\left(
x,\tau \right) \in \Omega _{2}\times 
\mathbb{R}
,  \tag{4.2}
\end{equation}%
\textit{here }$a_{2},$\textit{\ }$a_{3}$\textit{\ are nonnegative,
measurable functions defined on }$\Omega _{2}$\textit{.}

\textbf{(III)} \textit{On }$\Omega _{3}\times 
\mathbb{R}
$\textit{, }$a\left( x,\tau \right) $\textit{\ satisfies the inequality}%
\begin{equation}
a\left( x,\tau \right) \tau \geq a_{4}\left( x\right) \left\vert \tau
\right\vert ^{\xi \left( x\right) }-a_{5}\left( x\right) ,\text{ }\left(
x,\tau \right) \in \Omega _{3}\times 
\mathbb{R}
,  \tag{4.3}
\end{equation}%
\textit{here }$\xi $\textit{\ is the same function as in (4.1)\ and }$%
a_{4}\left( x\right) \geq \bar{A}_{0}>0$\textit{\ a.e. }$x\in \Omega _{3}$%
\textit{\ and }$a_{5}$\textit{\ is nonnegative, measurable function defined
on }$\Omega _{3}$\textit{.}

\bigskip

Let $p_{1}$ be a number holds $p\left( x\right) \geq p_{1}\geq 2$ a.e. $x\in
\Omega $, $q_{1}\equiv \frac{p_{1}}{p_{1}-1}$, $\gamma \left( x\right)
\equiv \frac{p\left( x\right) -p_{1}}{p_{1}-1}$ and $\theta \left( x\right)
\equiv \frac{\xi \left( x\right) +\gamma \left( x\right) }{\gamma \left(
x\right) +1}$, $x\in \Omega $.

We introduce the following class of functions:%
\begin{equation*}
\tilde{P}_{0}\equiv \left\{ u\in L^{1}\left( \Omega \right) \mid
\sum\limits_{i=1}^{n}\left\Vert \left\vert u\right\vert ^{\gamma \left(
x\right) }D_{i}u\right\Vert _{L^{p_{1}}\left( \Omega \right) }+\left\Vert
u\right\Vert _{L^{\frac{nq_{1}\left( p\left( x\right) -1\right) }{n-q_{1}}%
}\left( \Omega \right) }<\infty \right\} \cap
\end{equation*}%
\begin{equation*}
\cap L^{\xi \left( x\right) +\gamma \left( x\right) }\left( \Omega \right)
\cap \left\{ u\mid _{\partial \Omega }\equiv 0\right\}
\end{equation*}%
and%
\begin{equation*}
P_{0}\equiv \left\{ u\in L^{1}\left( \Omega \right) \mid
\sum\limits_{i=1}^{n}\left\Vert \left\vert u\right\vert ^{p\left( x\right)
-2}D_{i}u\right\Vert _{L^{q_{1}}\left( \Omega \right) }+\left\Vert
u\right\Vert _{L^{\frac{nq_{1}\left( p\left( x\right) -1\right) }{n-q_{1}}%
}\left( \Omega \right) }<\infty \right\} \cap
\end{equation*}%
\begin{equation*}
\cap L^{\xi \left( x\right) +\gamma \left( x\right) }\left( \Omega \right)
\cap \left\{ u\mid _{\partial \Omega }\equiv 0\right\}
\end{equation*}

\begin{notation}
From Lemma 4.2 and Remark 4.3, it follows that $\tilde{P}_{0}$ and $P_{0}$
are metric spaces (also pn-spaces). Furthermore $\phi _{0}\left( x,\tau
\right) \equiv \left\vert \tau \right\vert ^{p\left( x\right) -2}\tau $ is a
homeomorphism between $P_{0}$ and $W_{0}^{1,q_{1}}\left( \Omega \right) \cap
L^{\frac{\theta \left( x\right) }{p_{1}-1}}\left( \Omega \right) $ and $\phi
_{1}\left( x,\tau \right) \equiv \left\vert \tau \right\vert ^{\gamma \left(
x\right) }\tau $ is a homeomorphism between $\tilde{P}_{0}$ and $%
W_{0}^{1,p_{1}}\left( \Omega \right) \cap L^{\theta \left( x\right) }\left(
\Omega \right) $.
\end{notation}

We investigate (1.1) for $h\in W^{-1,q_{1}}\left( \Omega \right) +L^{\theta
^{\ast }\left( x\right) }\left( \Omega \right) $ ($\theta ^{\ast }$ is
conjugate function of $\theta $). Solution of problem (1.1) is understood in
the following sense.

\begin{definition}
A function $u\in P_{0}$ is called the weak solution of problem (1.1) if it
satisfies the equality 
\begin{equation*}
-\int\limits_{\Omega }\ \Delta \left( \left\vert u\right\vert
^{p(x)-2}u\right) wdx+\int\limits_{\Omega }a\left( x,u\right)
wdx=\int\limits_{\Omega }h\left( x\right) wdx
\end{equation*}%
for all $w\in W_{0}^{1,p_{1}}\left( \Omega \right) \cap L^{\theta \left(
x\right) }\left( \Omega \right) $.
\end{definition}

\bigskip

We define the following functions:%
\begin{equation*}
\left. 
\begin{array}{l}
\mu _{1}\left( x\right) \equiv \tfrac{p\left( x\right) +\gamma \left(
x\right) }{p\left( x\right) -\xi _{1}\left( x\right) }\text{, }x\in \Omega
_{2}\text{, }\mu _{2}\left( x\right) \equiv \tfrac{\xi _{1}\left( x\right)
+\gamma \left( x\right) }{\xi _{1}\left( x\right) }\text{, }x\in \Omega _{2}%
\text{ and }\mu _{3}\left( x\right) \equiv \tfrac{\xi \left( x\right)
+\gamma \left( x\right) }{\xi \left( x\right) }\text{, }x\in \Omega _{3} \\ 
\mu _{4}\left( x\right) \equiv \left\{ 
\begin{array}{l}
\theta ^{\ast }\left( x\right) \text{ \ \ if\ }x\in \Omega _{1} \\ 
q_{1}\text{ \ \ \ \ \ \ \ if\ }x\in \Omega _{2}\cup \Omega _{3}%
\end{array}%
\right. \text{, }\mu \left( x\right) \equiv \left\{ 
\begin{array}{c}
\frac{p_{1}\theta ^{\ast }\left( x\right) }{p_{1}-\theta \left( x\right) }%
\text{ \ \ \ \ \ if }x\in \Omega _{1} \\ 
\frac{\tilde{p}_{1}\theta ^{\ast }\left( x\right) }{\tilde{p}_{1}-\theta
\left( x\right) }\text{ \ \ \ \ \ if }x\in \Omega _{2} \\ 
\infty \text{ \ \ \ \ \ \ \ \ \ \ \ \ if }x\in \Omega _{3}%
\end{array}%
\right. \text{, }\tilde{p}_{1}\equiv \frac{np_{1}}{n-q_{1}} \\ 
\text{and }\tilde{p}\left( x\right) \equiv \tilde{p}_{1}\left( \gamma \left(
x\right) +1\right) -\gamma \left( x\right) \text{. Here }p_{1}\equiv \frac{%
p\left( x\right) +\gamma \left( x\right) }{\gamma \left( x\right) +1},\text{ 
}q_{1}\equiv \frac{p\left( x\right) +\gamma \left( x\right) }{p\left(
x\right) -1},\text{ }x\in \Omega%
\end{array}%
\right.
\end{equation*}

\begin{theorem}
Assume that \textbf{(I)-(III) }hold. If\ $a_{2}\in L^{\mu _{1}\left(
x\right) }\left( \Omega _{2}\right) ,$ $a_{3}\in L^{\mu _{2}\left( x\right)
}\left( \Omega _{2}\right) ,$ $a_{5}\in L^{\mu _{3}\left( x\right) }\left(
\Omega _{3}\right) ,$\ $a_{4}\in L^{\infty }\left( \Omega _{3}\right) ,$ $%
a_{1}\in L^{\mu _{4}\left( x\right) }\left( \Omega \right) $ and\ $a_{0}\in
L^{\mu \left( x\right) }\left( \Omega \right) $, then $\forall h\in
W^{-1,q_{1}}\left( \Omega \right) +L^{\theta ^{\ast }\left( x\right) }\left(
\Omega \right) $ problem (1.1) has a generalized solution in $P_{0}.\ $
\end{theorem}

\bigskip

As mentioned in introduction part, we investigate problem (1.1) such a
different method. In order to study that problem, firstly we transform it to
equivalent problem by the transformation $\phi _{1}:$ $u\rightarrow
\left\vert u\right\vert ^{\gamma \left( x\right) }u$. Following equality can
be obtained easily,%
\begin{equation*}
-\ \Delta \left( \left\vert u\right\vert ^{p(x)-2}u\right) =-\ \Delta \left(
\left\vert \left\vert u\right\vert ^{\gamma \left( x\right) }u\right\vert
^{p_{1}-2}\left\vert u\right\vert ^{\gamma \left( x\right) }u\right) =
\end{equation*}%
\begin{equation*}
=\left( p_{1}-1\right) \sum\limits_{i=1}^{n}-D_{i}\left( \left\vert
\left\vert u\right\vert ^{\gamma \left( x\right) }u\right\vert
^{p_{1}-2}D_{i}\left( \left\vert u\right\vert ^{\gamma \left( x\right)
}u\right) \right) .
\end{equation*}%
Once using the transformation $u\rightarrow \left\vert u\right\vert ^{\gamma
\left( x\right) }u$ at $a$, we sign the established function with $b$ i.e.%
\begin{equation*}
b\left( x,\left\vert u\right\vert ^{\gamma \left( x\right) }u\right) \equiv
a\left( x,u\right)
\end{equation*}%
it is clear that $b:$ $\Omega \times 
\mathbb{R}
\longrightarrow 
\mathbb{R}
$ is also Caratheodory function.

Also, it is obvious that 
\begin{equation*}
\left\vert u\right\vert ^{\gamma \left( x\right) }u\mid _{\partial \Omega
}=0\Leftrightarrow u\mid _{\partial \Omega }=0
\end{equation*}%
Consequently (1.1) can be written in the form;%
\begin{equation}
\left\{ 
\begin{array}{l}
\left( p_{1}-1\right) \sum\limits_{i=1}^{n}-D_{i}\left( \left\vert
\left\vert u\right\vert ^{\gamma \left( x\right) }u\right\vert
^{p_{1}-2}D_{i}\left( \left\vert u\right\vert ^{\gamma \left( x\right)
}u\right) \right) +b\left( x,\left\vert u\right\vert ^{\gamma \left(
x\right) }u\right) =h\left( x\right) \\ 
\left\vert u\right\vert ^{\gamma \left( x\right) }u\mid _{\partial \Omega }=0%
\end{array}%
\right.  \tag{4.4}
\end{equation}

Let sign $v\equiv \left\vert u\right\vert ^{\gamma \left( x\right) }u$ then
from (4.4) we establish the following 
\begin{equation}
\left\{ 
\begin{array}{l}
\left( p_{1}-1\right) \sum\limits_{i=1}^{n}-D_{i}\left( \left\vert
v\right\vert ^{p_{1}-2}D_{i}v\right) +b\left( x,v\right) =h\left( x\right)
\\ 
v\mid _{\partial \Omega }=0%
\end{array}%
\right.  \tag{4.5}
\end{equation}%
which is equivalent to (1.1) as a immediate consequence of Lemma 4.2.
Obviously, problem (4.5) is same as the problem (1.2) which we have studied
in Section 3.

\begin{lemma}
Under the conditions of Theorem 4.6 for $\forall h\in W^{-1,q_{1}}\left(
\Omega \right) +L^{\theta ^{\ast }\left( x\right) }\left( \Omega \right) $,
problem (4.5) has generalized solution, in the sense of definition 3.1, in
the space $\mathring{S}_{1,\left( p_{1}-2\right) q_{1},q_{1}}\left( \Omega
\right) \cap L^{\theta \left( x\right) }\left( \Omega \right) .$
\end{lemma}

\begin{proof}
For the proof we only need to prove that $b\left( x,v\right) $ in problem
(4.5) fulfills all the conditions of Theorem 3.2.

If we rewrite the inequality (4.1) in terms of $v,$ we have 
\begin{equation}
\left\vert b\left( x,v\right) \right\vert \leq a_{0}\left( x\right)
\left\vert v\right\vert ^{\theta \left( x\right) -1}+a_{1}\left( x\right) 
\text{ on }\Omega \times 
\mathbb{R}
\text{ }  \tag{4.6}
\end{equation}%
Also in terms of $\theta $ the sets $\Omega _{i},$ $i=1,2,3$ can be written
equivalently, by simple calculations, in the form which is given below i.e.
the inequalities which define the sets $\Omega _{i}$ in (*) are equivalent
to the ones given below.%
\begin{eqnarray*}
\Omega _{1} &=&\left\{ x\in \Omega \text{ }|\text{ }1\leq \theta ^{-}\leq
\theta \left( x\right) \leq p_{1}-\tilde{\eta}_{0}\right\} \\
\Omega _{2} &=&\left\{ x\in \Omega \text{ }|\text{ }p_{1}-\tilde{\eta}%
_{0}<\theta \left( x\right) \leq \tilde{p}_{1}\right\} \\
\Omega _{3} &=&\left\{ x\in \Omega \text{ }|\text{ }\tilde{p}_{1}\leq \theta
\left( x\right) \leq \theta ^{+}<\infty \right\}
\end{eqnarray*}%
where $\tilde{\eta}_{0}=\tilde{\eta}_{0}\left( \eta _{0}\right) >0$ is
sufficently small.

Now we prove that under the conditions of Theorem 4.6, conditions \textbf{%
(ii)} and \textbf{(iii)} of Theorem 3.2 hold. From (4.2), (4.3), we have the
following inequalities for $b$: 
\begin{equation}
b\left( x,v\right) v\geq -a_{2}\left( x\right) \left\vert v\right\vert ^{%
\frac{\xi _{1}\left( x\right) +\gamma \left( x\right) }{\gamma \left(
x\right) +1}}-a_{3}\left( x\right) \left\vert v\right\vert ^{\frac{\gamma
\left( x\right) }{\gamma \left( x\right) +1}}\text{ on }\Omega _{2}\times 
\mathbb{R}
\text{ }  \tag{4.7}
\end{equation}%
and 
\begin{equation}
b\left( x,v\right) v\geq a_{4}\left( x\right) \left\vert v\right\vert
^{\theta \left( x\right) }-a_{5}\left( x\right) \left\vert v\right\vert ^{%
\frac{\gamma \left( x\right) }{\gamma \left( x\right) +1}}\text{ on }\Omega
_{3}\times 
\mathbb{R}
\text{ .}  \tag{4.8}
\end{equation}%
Coefficents and exponents in (4.7) and (4.8) hold the followings: Since $\xi
_{1}\left( x\right) <p\left( x\right) $ a.e. $x\in \Omega _{2}$ so we have 
\begin{equation}
\tfrac{\xi _{1}\left( x\right) +\gamma \left( x\right) }{\gamma \left(
x\right) +1}<p_{1}\text{ a.e. }x\in \Omega _{2}\text{.}  \tag{4.9}
\end{equation}%
Applying Young's inequality to the second term in (4.7), we arrive at%
\begin{equation}
a_{3}\left( x\right) \left\vert v\right\vert ^{\frac{\gamma \left( x\right) 
}{\gamma \left( x\right) +1}}\leq \left( a_{3}\left( x\right) \right) ^{^{%
\frac{\xi _{1}\left( x\right) +\gamma \left( x\right) }{\xi _{1}\left(
x\right) }}}+\left\vert v\right\vert ^{^{\frac{\xi _{1}\left( x\right)
+\gamma \left( x\right) }{\gamma \left( x\right) +1}}}.  \tag{4.10}
\end{equation}%
Using $\epsilon $-Young's inequality, we estimate the second term in (4.8) 
\begin{equation}
a_{5}\left( x\right) \left\vert v\right\vert ^{\frac{\gamma \left( x\right) 
}{\gamma \left( x\right) +1}}\leq C\left( \epsilon \right) \left(
a_{5}\left( x\right) \right) ^{^{\frac{\xi \left( x\right) +\gamma \left(
x\right) }{\xi \left( x\right) }}}+\epsilon \left\vert v\right\vert ^{\theta
\left( x\right) }  \tag{4.11}
\end{equation}%
for sufficently small $\epsilon >0$ such that $a_{4}\left( x\right)
-\epsilon \geq \tilde{A}_{0}>0$. From (4.6)-(4.11), we get that under the
conditions of Theorem 4.6 the transformed problem (4.5) satisfies all
conditions of Theorem 3.2. Consequently from Theorem 3.2, (4.5) has a weak
solution in $\mathring{S}_{1,\left( p_{1}-2\right) q_{1},q_{1}}\left( \Omega
\right) \cap L^{\theta \left( x\right) }\left( \Omega \right) $ for $\forall
h\in W^{-1,q_{1}}\left( \Omega \right) +L^{\theta ^{\ast }\left( x\right)
}\left( \Omega \right) $.
\end{proof}

\begin{remark}
We prove that with Lemma 4.7, $\left\vert u\right\vert ^{\gamma \left(
x\right) }u$ is solution of Problem (4.4) and $\left\vert u\right\vert
^{\gamma \left( x\right) }u\in \mathring{S}_{1,\left( p_{1}-2\right)
q_{1},q_{1}}\left( \Omega \right) \cap L^{\theta \left( x\right) }\left(
\Omega \right) $
\end{remark}

\begin{lemma}
Let the conditions of Theorem 4.6 hold. If $\left\vert u\right\vert ^{\gamma
\left( x\right) }u\in \mathring{S}_{1,\left( p_{1}-2\right)
q_{1},q_{1}}\left( \Omega \right) \cap L^{\theta \left( x\right) }\left(
\Omega \right) $ then $u\in P_{0}$
\end{lemma}

\begin{proof}
Since $v=\left\vert u\right\vert ^{\gamma \left( x\right) }u\in $ $L^{\theta
\left( x\right) }\left( \Omega \right) $ and $\sigma _{\theta }\left(
v\right) =\sigma _{\xi +\gamma }\left( u\right) $ thus we have $u\in L^{\xi
\left( x\right) +\gamma \left( x\right) }\left( \Omega \right) $.

As we have the embedding (Theorem 2.6) 
\begin{equation*}
\mathring{S}_{1,\left( p_{1}-2\right) q_{1},q_{1}}\left( \Omega \right)
\subset L^{\tilde{p}_{1}}\left( \Omega \right)
\end{equation*}%
thus we conclude that $\left\vert u\right\vert ^{\gamma \left( x\right)
}u\in L^{\tilde{p}_{1}}\left( \Omega \right) $ so from this and the
definition of $\gamma $, we get that $u\in L^{\frac{nq_{1}\left( p\left(
x\right) -1\right) }{n-q_{1}}}\left( \Omega \right) $.

Moreover using this fact and Lemma 4.1, one can see that%
\begin{equation*}
\left\vert u\right\vert ^{p\left( x\right) -2}u\ln \left\vert u\right\vert
=\left\vert u\right\vert ^{\gamma \left( x\right) +\left( \gamma \left(
x\right) +1\right) \left( p_{1}-2\right) }u\ln \left\vert u\right\vert \in
L^{q_{1}}\left( \Omega \right) .
\end{equation*}%
Also from the definition of $\mathring{S}_{1,\left( p_{1}-2\right)
q_{1},q_{1}}\left( \Omega \right) $, it follows that $v\in \mathring{S}%
_{1,\left( p_{1}-2\right) q_{1},q_{1}}\left( \Omega \right) \Leftrightarrow
\left\vert v\right\vert ^{p_{1}-2}D_{i}v\in L^{q_{1}}\left( \Omega \right) $%
. Using these and the definition of $v$, we obtain the following equality
for $\forall w\in L^{p_{1}}\left( \Omega \right) $: 
\begin{equation*}
\langle \left\vert v\right\vert ^{p_{1}-2}D_{i}v,w\rangle -\langle \left(
D_{i}\gamma \right) \left\vert u\right\vert ^{\gamma \left( x\right) +\left(
\gamma \left( x\right) +1\right) \left( p_{1}-2\right) }u\ln \left\vert
u\right\vert ,w\rangle =
\end{equation*}%
\begin{equation*}
=\langle \left( 1+\gamma \left( x\right) \right) \left\vert u\right\vert
^{\gamma \left( x\right) +\left( \gamma \left( x\right) +1\right) \left(
p_{1}-2\right) }D_{i}u,w\rangle
\end{equation*}%
Using H\"{o}lder inequality to the left of that equality, we have%
\begin{equation*}
C\left( \left\Vert \left\vert v\right\vert ^{p_{1}-2}D_{i}v\right\Vert
_{q_{1}}\left\Vert w\right\Vert _{p_{1}}+\left\Vert \left\vert u\right\vert
^{p\left( x\right) -1}\ln \left\vert u\right\vert \right\Vert
_{q_{1}}\left\Vert w\right\Vert _{p_{1}}\right) \geq
\end{equation*}%
\begin{equation*}
\geq \left\vert \langle \left( 1+\gamma \left( x\right) \right) \left\vert
u\right\vert ^{\gamma \left( x\right) +\left( \gamma \left( x\right)
+1\right) \left( p_{1}-2\right) }D_{i}u,w\rangle \right\vert
\end{equation*}%
where $C=C\left( \left\Vert \gamma \right\Vert _{C^{1}\left( \Omega \right)
}\right) >0$. From last inequality, we obtain%
\begin{equation*}
\left\Vert \left\vert v\right\vert ^{p_{1}-2}D_{i}v\right\Vert
_{q_{1}}+\left\Vert \left\vert u\right\vert ^{p\left( x\right) -1}\ln
\left\vert u\right\vert \right\Vert _{q_{1}}\geq
\end{equation*}%
\begin{equation*}
\geq \tilde{C}\left\Vert \left\vert u\right\vert ^{\gamma \left( x\right)
+\left( \gamma \left( x\right) +1\right) \left( p_{1}-2\right)
}D_{i}u\right\Vert _{q_{1}}=\tilde{C}\left\Vert \left\vert u\right\vert
^{p\left( x\right) -2}D_{i}u\right\Vert _{q_{1}}
\end{equation*}%
hence all in all we get that $u\in P_{0}$.
\end{proof}

\bigskip

Now we give the proof of Theorem 4.6.

\begin{proof}
\textbf{(of Theorem 4.6) }For proof we use Theorem 2.9 in general form. We
introduce the following spaces and mappings in order to apply this theorem
to prove Theorem 4.6.%
\begin{equation*}
S_{0}\equiv P_{0},\text{ }Y\equiv W^{-1,q_{1}}\left( \Omega \right)
+L^{\theta ^{\ast }\left( x\right) }\left( \Omega \right) ,\text{ }%
X_{0}\equiv \tilde{P}_{0}
\end{equation*}%
and%
\begin{equation*}
Y_{0}^{\ast }\equiv Y^{\ast }\equiv W_{0}^{1,p_{1}}\left( \Omega \right)
\cap L^{\theta \left( x\right) }\left( \Omega \right)
\end{equation*}%
\begin{equation*}
f:S_{0}\longrightarrow Y
\end{equation*}%
\begin{equation*}
f\left( u\right) \equiv -\ \Delta \left( \left\vert u\right\vert
^{p(x)-2}u\right) +a\left( x,u\right)
\end{equation*}%
\begin{equation*}
g:X_{0}\subset S_{0}\longrightarrow Y^{\ast }
\end{equation*}%
\begin{equation*}
g\left( u\right) \equiv \left\vert u\right\vert ^{\gamma \left( x\right) }u
\end{equation*}%
To apply this theorem we have to show the conditions of Theorem 2.9 is hold.
Weak compactness and boundness of $f:P_{0}\longrightarrow W^{-1,q_{1}}\left(
\Omega \right) +L^{\theta ^{\ast }\left( x\right) }\left( \Omega \right) $
follows from Notation 4.4, Lemma 4.7 and Lemma 4.9 by virtue of Lemma 3.4
and 3.5.

For $u\in \tilde{P}_{0}$ we have the equality 
\begin{equation*}
\langle f\left( u\right) ,g\left( u\right) \rangle =\langle -\ \Delta \left(
\left\vert u\right\vert ^{p(x)-2}u\right) +a\left( x,u\right) ,\left\vert
u\right\vert ^{\gamma \left( x\right) }u\rangle =
\end{equation*}%
\begin{equation*}
=\langle -\ \Delta \left( \left\vert \left\vert u\right\vert ^{\gamma \left(
x\right) }u\right\vert ^{p_{1}-2}\left\vert u\right\vert ^{\gamma \left(
x\right) }u\right) +b\left( x,\left\vert u\right\vert ^{\gamma \left(
x\right) }u\right) ,\left\vert u\right\vert ^{\gamma \left( x\right)
}u\rangle
\end{equation*}%
\begin{equation*}
=\langle \left( p_{1}-1\right) \sum\limits_{i=1}^{n}-D_{i}\left( \left\vert
\left\vert u\right\vert ^{\gamma \left( x\right) }u\right\vert
^{p_{1}-2}D_{i}\left( \left\vert u\right\vert ^{\gamma \left( x\right)
}u\right) \right) ,\left\vert u\right\vert ^{\gamma \left( x\right)
}u\rangle +
\end{equation*}%
\begin{equation*}
+\langle b\left( x,\left\vert u\right\vert ^{\gamma \left( x\right)
}u\right) ,\left\vert u\right\vert ^{\gamma \left( x\right) }u\rangle =
\end{equation*}%
\begin{equation*}
=\langle \left( p_{1}-1\right) \sum\limits_{i=1}^{n}\left( \left\vert
\left\vert u\right\vert ^{\gamma \left( x\right) }u\right\vert
^{p_{1}-2}D_{i}\left( \left\vert u\right\vert ^{\gamma \left( x\right)
}u\right) \right) ,D_{i}\left( \left\vert u\right\vert ^{\gamma \left(
x\right) }u\right) \rangle +
\end{equation*}%
\begin{equation*}
+\langle b\left( x,\left\vert u\right\vert ^{\gamma \left( x\right)
}u\right) ,\left\vert u\right\vert ^{\gamma \left( x\right) }u\rangle .
\end{equation*}%
Generating a "coercive\ pair" of the mappings $f$ and $g$ on $\tilde{P}_{0}$
follows from the above equality and Lemma 4.7 by virtue of Lemma 3.3.

Also as $g:\tilde{P}_{0}\subset P_{0}\longrightarrow W_{0}^{1,p_{1}}\left(
\Omega \right) \cap L^{\theta \left( x\right) }\left( \Omega \right) $ is
bounded and fulfills the conditions of (2) in Theorem 2.9. Thus we show that
mappings $f$ and $g$ satisfiy all the conditions of Theorem 2.9.
Consequently applying that to problem (1.1), we obtain that $\forall h\in
W^{-1,q_{1}}\left( \Omega \right) +L^{\theta ^{\ast }\left( x\right) }\left(
\Omega \right) $ the equation%
\begin{equation*}
-\int\limits_{\Omega }\ \left[ \Delta \left( \left\vert u\right\vert
^{p(x)-2}u\right) +a\left( x,u\right) \right] wdx=\int\limits_{\Omega
}h\left( x\right) wdx\text{, }w\in W_{0}^{1,p_{1}}\left( \Omega \right) \cap
L^{\theta \left( x\right) }\left( \Omega \right)
\end{equation*}%
is solvable in $P_{0}$.
\end{proof}

\bigskip {\large \noindent \textbf{References}}

\bigskip \noindent[1] E. Acerby, G. Mingione: \emph{Regularity results for
stationary electro-rheological liquids}. Arch. Ration. Mech. Anal. 164
(2002), 213--259.

\noindent[2] T. Adamowicz, T. Olli: \emph{H\"{o}lder continuity of
quasiminimizers with nonstandard growth}. Nonlinear Anal. 125 (2015),
433-456.

\noindent[3] R. A. Adams: \emph{Sobolev Spaces.} Academic Press, New York,
(1975).

\noindent[4] S. N. Antontsev, S. I. Shmarev: \emph{A model porous medium
equation with variable exponent of nonlinearity: existence, uniqueness and
localization properties of solutions}. Nonlinear Anal. 60 (2005), 515--545.

\noindent[5] S. N. Antontsev, J. F. Rodrigues: \emph{On stationary
thermo-rheological viscous flows.} Ann. Univ. Ferrara, Sez. VII, Sci. Mat.
52 (1) (2006), 19-36.

\noindent[6] S. N. Antontsev, S. I. Shmarev: \emph{On the localization of
solutions of elliptic equations with nonhomogeneous anisotropic degeneration}%
. (Russian) Sibirsk. Mat. Zh. 46., (2005), no. 5, 963--984; translation in
Siberian Math. J. 46., no. 5 (2005), 765--782.

\noindent[7] S. N. Antontsev, S. I. Shmarev: \emph{Elliptic equations and
systems with nonstandard growth conditions: Existence, uniqueness and
localization properties of solutions}. Nonlinear Anal. 65 (2006), 728-761.

\noindent[8] L. Diening: \emph{Theoretical and numerical results for
electrorheological fluids}. Ph.D. Thesis, (2002)

\noindent[9] Y. Chen, S. Levine, M. Rao: \emph{Variable exponent, linear
growth functionals in image restoration}. SIAM J. Appl. Math. 66 (2006),
1383-1406.

\noindent[10] Yu. A. Dubinskii: \emph{Weak convergence in nonlinear elliptic
and parabolic equations.} Mat.Sb. (1965), 67(109).

\noindent[11] X. Fan, D. Zhao: \emph{On the spaces }$L^{p\left( x\right) }(%
\Omega
)$\emph{\ and }$W^{k,p\left( x\right) }(%
\Omega
)$\emph{. }J. Math. Anal. Appl. 263 (2001), 424--446.

\noindent[12] X. Fan, J. Shen, D. Zhao: \emph{Sobolev embedding theorems for
spaces }$W^{k,p\left( x\right) }(%
\Omega
)$. J. Math. Anal. Appl. 262, no. 2 (2001), 749--760.

\noindent[13] X. Fan, D. Zhao: \emph{The quasi-minimizer of integral
functionals with }$\emph{m(x)}$\emph{\ growth conditions}. Nonlinear Anal.
39 (7) (2000), 807-816.

\noindent[14] S. Fucik, A. Kufner:\emph{\ Nonlinear Differential Equations}.
Elsevier, New York (1980)

\noindent[15] M. Giaquinta: \emph{Growth conditions and regularity, a
counterexample}. Manuscripta. Math. 59 (1987), 245-248.

\noindent[16] H. Hudzik: \emph{On generalized Orlicz--Sobolev space. }Funct.
Approx. Comment. Math. 4 (1976), 37--51.

\noindent[17] O. Kovacik, J. Rakosnik: \emph{On spaces }$L^{p\left( x\right)
}$\emph{\ and }$W^{k,p\left( x\right) }$\emph{.} Czechoslovak Math. J. 41
(1991), 592--618.

\noindent[18] J. L. Lions: \emph{Queques methodes de resolution des
problemes aux limites non lineaires. }Dunod and Gauthier-Villars, Paris
(1969)

\noindent[19] P. Marcellini: \emph{Regularity of minimizers of integrals of
the calculus of variations with nonstandard growth conditions}. Arch.
Rational Mech. Anal. 105 (1989), 267-284.

\noindent[20] G. de Marsily: \emph{Quantitative Hydrogeology. Groundwater
Hydrology for Engineers}. Academic Press, London (1986)

\noindent[21] J. Musielak:\emph{\ Orlicz Spaces and Modular Spaces}. Lecture
Notes in Mathematics, Vol.1034, Springer-Verlag, Berlin (1983)

\noindent[22] V. R\u{a}dulescu, D. Repov\v{s}: \emph{Partial Differential
Equations with Variable Exponents: Variational methods and Quantitative
Analysis}, CRC Press, Taylor \& Francis Group, Boca Raton FL, (2015).

\noindent[23] M. M. Rao: \emph{Measure Theory and Integration. }John Wiley
\& Sons, New York (1984)

\noindent[24] P. A. Raviart: \emph{Sur la resolution et l'approximation de
certaines equations paraboliques non lineaires.} Arch. Rational Mech. Anal.
25 (1967), 64-80.

\noindent[25] M. Ruzicka: \emph{Electrorheological Fluids: Modeling and
Mathematical Theory. }In Lecture Notes in Mathematics, vol. 1748 Springer,
Berlin (2000).

\noindent[26] K. N. Soltanov, J. Sprekels: \emph{Nonlinear equations in
non-reflexive Banach spaces and strongly nonlinear equations. }Adv. Math.
Sci. Appl. 9, no. 2 (1999), 939-972.

\noindent[27] K. N. Soltanov, M. A. Ahmadov: \emph{On nonlinear parabolic
equation in nondivergent form with implicit degeneration and embedding
theorems.} arXiv:1207.7063v1 [math.AP] (2012)

\noindent[28] K. N. Soltanov: \emph{Some imbedding theorems and nonlinear
differential equations. }Trans. Acad. Sci. Azerb. Ser. Phys.-Tech. Math.
Sci. 19 (1999), no. 5, Math. Mech. (2000), 125--146. (Reviewer: H. Triebel)

\noindent[29] K. N. Soltanov: \emph{Some nonlinear equations of the
nonstable filtration type and embedding theorems. }Nonlinear Anal. 65
(2006), 2103-2134.

\noindent[30] K. N. Soltanov: \emph{On noncoercive semilinear equations.}
Nonlinear Anal. Hybrid Syst. 2 , no. 2 (2008), 344--358.

\noindent[31] K. N. Soltanov: \emph{"Some applications of nonlinear analysis
to differential equations"}, Baku, ELM, 292pp(in Russian), (2002)

\noindent[32] K. N. Soltanov: \emph{Some embedding theorems and its
applications to nonlinear equations. }Differensial'nie Uravnenia. 20, 12
(1984), 2181-2184; English transl. in Differential Equations. 20 (1984).

\noindent[33] K. N. Soltanov: \emph{On some modification Navier-Stokes
equations. }Nonlinear Anal. 52, no. 3 (2003), 769-793.

\noindent[34] V. V. Zhikov: \emph{On some variational problems}. Russian J.
Math. Phys., 5:1 (1997), 105-116.

\noindent[35] V. V. Zhikov: \emph{On the technique for passing to the limit
in nonlinear elliptic equations}. Functional Anal. and Its App. Vol. 43, No.
2 (2009), 96-112.

\noindent[36] V. V. Zhikov: \emph{Averaging of functionals of the calculus
of variations and elasticity theory. }Math. USSR. Izv. 29 (1987), 33-36.

\end{document}